\documentclass[11pt]{article}
\usepackage{fullpage,amsmath, amsfonts, amsthm,xr}
\usepackage{graphicx, algorithmic, algorithm, enumerate,bm}
\usepackage{natbib}
\usepackage[bbgreekl]{mathbbol}
\usepackage{csquotes,caption}
\usepackage{multirow}
\usepackage[toc,page]{appendix}
\usepackage{hyperref}

\newenvironment{theorem}[2][Theorem]{\begin{trivlist}
\item[\hskip \labelsep {\bfseries #1}\hskip \labelsep {\bfseries #2.}]}{\end{trivlist}}
\newenvironment{lemma}[2][Lemma]{\begin{trivlist}
\item[\hskip \labelsep {\bfseries #1}\hskip \labelsep {\bfseries #2.}]}{\end{trivlist}}

\begin{document}
\begin{center}
{\bf\LARGE Local Quadratic Estimation of the Curvature in a Functional Single Index Model\footnote{Supported in part by NSF Grant DEB-1353039 and DMS-1407600}}\\
\vspace{0.5cm}
{\sc Zi Ye\footnote{PhD Candidate, Department of Statistical Sciences, Cornell University, zy234@cornell.edu}, Giles Hooker\footnote{Associate Professor, Department of Statistical Sciences, Cornell University, gjh27@cornell.edu}}\\
\end{center}

\begin{abstract}
The nonlinear effects of environmental variability on species abundance plays an important role in the maintenance of ecological diversity. Nonetheless, many common models use parametric nonlinear terms pre-determining ecological conclusions.  Motivated by this concern, we study the estimate of the second derivative (curvature) of the link function $g$ in a functional single index model: $Y = g\left(\int X\left(t\right)\beta^0\left(t\right)\mathrm{d}t\right) + \epsilon$. Since the coefficient function $\beta^0$ and the link function $g$ are both unknown, the estimate is expressed as a nested optimization. For a fixed and unknown $\beta$, the link function $g$ and $g''$ are estimated by local quadratic approximation, then the coefficient function $\beta^0$ is estimated by minimizing the MSE of the model. In this paper, we derive the rate of convergence of the estimation is $\frac{1}{n}\sum\limits_{i=1}^n\mathrm{E}\left[\hat{g}''\left(\int X_i\hat{\beta}\right) - g''\left(\int X_i\beta^0\right)\right]^2 = O\left(h_n^4 + \frac{1}{nh_n^4}\right)$, where $h_n$ is the bandwidth in the local quadratic approximation. In addition, we prove that the argument of $g$, $\int X\beta^0$, can be estimated root-$n$ consistently. However, practical implementation of the method requires solving a nonlinear optimization problem, and our results show that the estimates of the link function and the coefficient function are quite sensitive to the choices of starting values.
\end{abstract}

\section{Introduction}
\subsection{Ecological Motivation}
Within mathematical ecology, nonlinear responses to environmental variability play an important role in maintaining the diversity of competing species. Species competing for the same resources can nonetheless co-exist by exploiting differing environmental conditions; see \cite{hutchinson1961paradox}, \cite{chesson1981environmental} and \cite{ellner1987alternate}. For an individual species, environmental fluctuation can accelerate growth rate \citep{drake2005population,koons2009life} or sometimes decrease long-term population growth rates \citep{lewontin1969population}. The nonlinearity of these responses also plays an important role in forecasting the effect of increased environmental variability under climate change. The motivating data for this study come from long-term observations of communities of prairie plants in which {\em Artemesia Triparta} -- sage brush -- is a dominant species and we wish to understand its responses to climate given by temperature and precipitation. \\
Traditional statistical models for plant growth make parametric assumptions that imply specific forms of nonlinearity, particularly in the presence of high-dimensional covariates. Instead, we use a nonparametric growth model of an individual plant or animal: 
\begin{align}
\bm{G} = g\left(\bm{E}\right) + \epsilon, \nonumber
\end{align}
where $\bm{G}$ and $\bm{E}$ are the growth and environment of a plant, $g$ is a link function to be estimated and $\epsilon$ is the random error. To answer the ecological question, \enquote{Would the growth be higher if we just gave the plant a constant environment at the average of $\bm{E}$?}, we need to compare $g\left[\mathrm{E}\left(\bm{E}\right)\right]$ and $\mathrm{E}\left[g\left(\bm{E}\right)\right]$.\\
If the link function $g$ is convex, $g\left[\mathrm{E}\left(\bm{E}\right)\right] \leq \mathrm{E}\left[g\left(\bm{E}\right)\right]$ by Jensen's inequality, and the plant grows better in a varying environment. Otherwise, if the link function $g$ is concave, a constant environment is preferred. Assuming a smooth function $g$, convexity is equivalent to $g''\left(s\right) > 0$, for all $s$ in the domain of $g$. Therefore, in this paper we consider the problem of estimating the curvature of the link function $g$. \\
To finalize this model, the environment $\bm{E}$ is described by the recent  history  of temperature and rainfall recorded at up to daily resolution. Since plants may be impacted by climate events over a long period of time \citep[see][]{dahlgren2011incorporating,clark2011climate}, we will consider the past two years of data. Following \cite{teller2016linking}, these are thought of as functional covariates leading to a representation of $\bm{E}$ as a functional linter term:
\begin{align}
\bm{E} = \int_0^1 X\left(t\right)\beta^0\left(t\right)\mathrm{d}t, \nonumber
\end{align}
where $\beta^0\left(t\right)$ is the coefficient function to be estimated, and $X\left(t\right)$ is the covariate function we observed, typically a measurement of climate history. \\
The growth model of a plant is now given by
\begin{align}
Y = g\left(\int_0^1 X\left(t\right)\beta^0\left(t\right)\mathrm{d}t\right) + \epsilon. \nonumber
\end{align}
This is Functional Single Index model, introduced in \cite{chen2011single} and \cite{ma2016estimation}. \\
In functional data analysis, a functional linear model (FLM) is defined as
\begin{align}
Y = \int_0^1 X\left(t\right)\beta^0\left(t\right)\mathrm{d}t + \epsilon, \nonumber
\end{align}
which is often used in modeling the relationship between a functional covariate and a scalar response. To assess curvature, we need a more flexible model than the FLM. A generalized functional linear model (GFLM) is proposed in \cite{muller2005generalized}, \cite{james2002generalized} and \cite{escabias2007functional}, which has the same form as the functional single index model but with a known link function $g$. The functional single index model could be considered as an extension to the GFLM, as it is more flexible and could model a variety of real-world data. \\
Compared to the generalized functional linear model, estimation of the link function $g$ based on a unknown coefficient function $\beta^0$ is challenging. Even if $\beta^0$ is known, estimating the second derivative of a nonparametric function directly is difficult. In this paper, we prove a theoretical convergence rate for an estimate of $g''$ in the functional single index model, even if there are some bias in estimating the coefficient function $\beta^0$. \\
The convergence rates that we derive are based on finding a global solution to a nonlinear optimization problem using a bandwidth that decreases at a known rate with $n$.  However, this requires overcoming several practical issues. First, to find an optimum for $\beta^0$, we rely on  nonlinear optimization methods which require an initial value from which to search for a minimum. Our experiments demonstrate that the performance of the estimate can depend critically on this choice of initial condition and natural choices which provide good estimates of $g$ do not necessarily work well for $g''$. Further, the optimal choice of bandwidth can be quite different between estimation targeting $g$ and that targeting $g''$ and we provide a heuristic post-cross-validation modification to improve the estimate of bandwidth. We expect similar rates of convergence will hold for alternative non-parametric estimators of $g$, penalized splines, for example, but that the specifics of smoothing parameter selection and nonlinear optimization can be expected to be quite different.  A detailed analysis of the optimization problem is beyond the scope of this paper.

\subsection{Previous Results}
In this section, we will introduce previous theoretical and empirical results for Single Index model and Functional Single Index model.
\subsubsection{Single Index Model}
There has been considerable research on the single index model, where the coefficient $\bm{\beta}^0$ is finite dimensional. The Single Index Model is defined as
\begin{align}
Y = g\left(\bm{X}\bm{\beta}^0\right) + \epsilon, \nonumber
\end{align}
where $\bm{X}$ is the covariate and $\bm{\beta}^0$ is the coefficient vector. There are three methods to estimate the link function $g$ and the coefficient vector $\bm{\beta}^0$. The Projection Pursuit Regression (PPR) approach introduced in \cite{hardle1993optimal} is a two-step estimation procedure:
\begin{enumerate}
\item Estimate the link function $g$ by the kernel method 
\begin{align}
\hat{g}_i\left(\bm{X}_i\bm{\beta}\left|\bm{\beta},h\right.\right) = \frac{\sum_{j\neq i} Y_jK\left(\frac{\bm{X}_i\bm{\beta}-\bm{X}_j\bm{\beta}}{h}\right)}{\sum_{j\neq i} K\left(\frac{\bm{X}_i\bm{\beta}-\bm{X}_j\bm{\beta}}{h}\right)}, \nonumber
\end{align}
where $h$ is the bandwidth.
\item Estimate the coefficient $\bm{\beta}^0$ by minimizing the mean squared error 
\begin{align}
\hat{S}\left(\bm{\beta},h\right) = \sum\limits_{i=1}^n \left[Y_i - \hat{g}_i\left(\bm{X}_i\bm{\beta}\left|\bm{\beta},h\right.\right)\right]^2. \nonumber
\end{align}
\end{enumerate}
\cite{hardle1993optimal} proved that the coefficient vector $\bm{\beta}^0$ can be estimated root-$n$ consistently. \cite{ichimura1993semiparametric} showed the asymptotic normality of the estimator. The other two approaches provide new methods to estimate the coefficient vector. The Average Derivative approach in \cite{hristache2001direct} showed that
\begin{align}
\mathrm{E}\left[\frac{\partial g\left(\bm{X\beta}\right)}{\partial \bm{X}}\right] \stackrel{z \doteq \bm{X}\bm{\beta}}{=} \mathrm{E}\left[\frac{\partial g\left(z\right)}{\partial z}\bm{\beta}\right] = \mathrm{E}\left[\frac{\partial g\left(z\right)}{\partial z}\right]\bm{\beta} \doteq \gamma\bm{\beta}. \nonumber
\end{align}
If we could find a consistent estimator of the average derivative $\mathrm{E}\left[\frac{\partial g\left(\bm{X\beta}^0\right)}{\partial \bm{X}}\right]$, we can get a consistent estimator of the coefficient $\bm{\beta}^0$ up to a scale. Normally, we require the coefficient vector to be norm $1$. \cite{stoker1986consistent} proposed two consistent estimators of the average derivative. \\
The sliced inverse regression method in \cite{li1991sliced} considered the estimation of the coefficient vector as a dimension-reduction problem. Any linear combination of the coefficient vector $\bm{\beta}^0$ is assumed to be an effective dimension-reduction (EDR) direction. They conduct a principle component analysis on the inverse regression space $\mathrm{E}\left(\bm{X}\left|Y\right.\right)$, and estimate the coefficient vector $\bm{\beta}^0$ by the largest component.

\subsubsection{Functional Single Index Model}
There are only a few papers in the functional single index model. In \cite{chen2011single}, similar to the projection pursuit regression in the single index model, the link function $g$ and the coefficient function $\beta^0$ are estimeated by a two-step procedure. The coefficient function $\beta^0$ is reduced to a finite dimensional coefficient vector by a spline basis. Under some assumptions, \cite{chen2011single} showed that 
\begin{align}
\frac{1}{n} \sum\limits_{j=1}^n\left[\hat{g}\left(\int X_j\hat{\beta}\right) - g\left(\int X_j\beta^0\right)\right]^2 = O\left(n^{-c}\right), \nonumber
\end{align}
for $c > 0$. In \cite{ma2016estimation}, two spline basis were used to represent the coefficient function and the link function, respectively, and the MSE was minimized iteratively until convergence. \cite{ma2016estimation}  constructed a asymptotic simultaneous confidence band for the coefficient function $\beta^0$. Our estimates follow \cite{chen2011single} but will examine the properties of $\hat{g}''$. By a clever decomposition of squared error, \cite{chen2011single} were able to avoid the need to directly account for the estimate of $\beta^0$. Unlike that case, to examine $g''$ we will need to obtain the $\sqrt{n}$ convergence rate for $\int X\beta^0$ directly, before we can examine our target.

\section{Estimation Procedure}
Suppose that we observe $n$ environment histories and responses $\left(X_1\left(t\right),Y_1\right), \cdots, \left(X_n\left(t\right),Y_n\right)$, independent and identically distributed as $\left(X\left(t\right), Y\right)$, where $t \in \left[0,1\right]$, with
\begin{align}
Y = g\left(\int_0^1 X\left(t\right)\beta^0\left(t\right)\mathrm{d}t\right) + \epsilon, \nonumber
\end{align}
where $Y$ is the scalar response variable, and $X\left(t\right)$ is the covariate function. For the purpose of simplification, we assume that the predicator $X$ and the coefficient function $\beta^0$ are defined in the domain $\left[0,1\right]$, and $\epsilon$ is a Gaussian random error. \\
To answer our ecological question, we are interested in estimating the second derivative (curvature) of the link function $g$. The estimate of the coefficient function $\beta^0$ is denoted as $\hat{\beta}$. Define a Hilbert space $\mathcal{B}$ as the set of the coefficient functions $\beta$, where $\hat{\beta}\left(t\right), \beta^0\left(t\right) \in \mathcal{B}$. \\
To estimate $\beta^0$, $g$ and $g''$, we use a local quadratic approximation. By Taylor's expansion, at a fixed point $x$, the link function $g$ can be approximated by
\begin{align}
g\left(x\right) \approx g\left(x_0\right) + g'\left(x_0\right)\left(x-x_0\right) + g''\left(x_0\right)\frac{\left(x-x_0\right)^2}{2}. \nonumber
\end{align}
Fix $u$, where $u$ is in the domain of the link function $g$, the curvature, denoted as $\hat{g}''$, is estimated by minimizing the weighted sum of squares
\begin{align}
&\left(\hat{a},\hat{b},\hat{c}\right) = \label{a}\\
&\inf_{a,b,c}\sum\limits_{i=1}^n \left\{\left[Y_i - a - b\left(\int_0^1 X_i\left(t\right)\beta\left(t\right)\mathrm{d}t - u\right) - c\frac{\left(\int_0^1 X_i\left(t\right)\beta\left(t\right)\mathrm{d}t - u\right)^2}{2}\right]^2K\left(\frac{\int_0^1 X_i\left(t\right)\beta\left(t\right)\mathrm{d}t - u}{h_n}\right)\right\},\nonumber
\end{align}
where $K$ is a kernel function and $h_n$ is the bandwidth. The estimators are then $\left(\hat{g}\left(u\right), \hat{g}'\left(u\right), \hat{g}''\left(u\right)\right) = \left(\hat{a}, \hat{b}, \hat{c}\right)$. \\
The coefficient function $\beta^0\left(t\right)$ is unknown in the penalized weighted sum of square $\left(\ref{a}\right)$. We estimate it by minimizing the MSE 
\begin{align}
\hat{\beta} = \inf_{\beta \in \mathcal{B}} \sum\limits_{i=1}^n \left[Y_i-\hat{g}\left(\int X_i\beta\right)\right]^2,\label{step1}
\end{align}
where 
\begin{align}
\hat{g}\left(\int X_i\beta\right) = \frac{\sum\limits_{j:j \neq i} Y_jK\left(\frac{\int X_i\beta - \int X_j\beta}{h_n}\right)}{\sum\limits_{j:j \neq i} K\left(\frac{\int X_i\beta - \int X_j\beta}{h_n}\right)}.\label{step2}
\end{align}
Since the kernel function $K$ is only defined in $\left[-1,1\right]$, we constraint the domain of the estimate of $g$ or $g''$ to be in $\left[-1,1\right]$ by normalizing the coefficients of $g$ or $g''$ under an orthonormal basis to be $1$ after optimization procedure.\\
Denote a column vector $\bm{Y} = \left(Y_1,\cdots,Y_n\right)^{\top}$. Fix $j \in \left\{1,\cdots,n\right\}$, and $\beta \in \mathcal{B}$, the estimated $\hat{g}\left(\int X_j\beta\right)$, $\hat{g}'\left(\int X_j\beta\right)$ and $\hat{g}''\left(\int X_j\beta\right)$ can be calculated as
\begin{align}
\left(\hat{g}\left(\int X_j\beta\right), \hat{g}'\left(\int X_j\beta\right), \hat{g}''\left(\int X_j\beta\right)\right)^{\top} = \left(\mathbb{X}_{\beta,j}^{\top}\mathbb{K}_{\beta,j}\mathbb{X}_{\beta,j}\right)^{-1}\left(\mathbb{X}_{\beta,j}^{\top}\mathbb{K}_{\beta,j}\right)\bm{Y},\label{step3}
\end{align}
where
\begin{eqnarray}
\hat{g}'\left(\int X_j\beta\right) &=& \left(\mathbb{X}_{\beta,j}^{\top}\mathbb{K}_{\beta,j}\mathbb{X}_{\beta,j}\right)^{-1}_{2\cdot}\left(\mathbb{X}_{\beta,j}^{\top}\mathbb{K}_{\beta,j}\right)\bm{Y} \doteq S_1\left(\beta;j\right)\bm{Y},\label{s1} \\
\hat{g}''\left(\int X_j\beta\right) &=& \left(\mathbb{X}_{\beta,j}^{\top}\mathbb{K}_{\beta,j}\mathbb{X}_{\beta,j}\right)^{-1}_{3\cdot}\left(\mathbb{X}_{\beta,j}^{\top}\mathbb{K}_{\beta,j}\right)\bm{Y} \doteq S_2\left(\beta;j\right)\bm{Y}, \label{s2}
\end{eqnarray}
where $\mathbb{A}_{k\cdot}$ denotes the $k^{\text{th}}$ row of a matrix $\mathbb{A}$, and the $\left(n \times 3\right)$-dimensional matrix $\mathbb{X}_{\beta,j}$ is
\begin{align}
\mathbb{X}_{\beta,j} = \left(\bm{1}, \int \bm{X}\beta - \left(\int X_j\beta\right)\bm{1}, \frac{\left(\int \bm{X}\beta - \left(\int X_j\beta\right)\bm{1}\right)^2}{2}\right), \nonumber
\end{align}
with $\int \bm{X}\beta \doteq \left(\int X_1\beta,\cdots,\int X_n\beta\right)^{\top}$, $\bm{1}$ is a $n$-dimensional column vector of ones, and the $\left(n \times n\right)$-dimensional matrix $\mathbb{K}_{\beta,j}$ is
\begin{align}
\mathbb{K}_{\beta,j} = \text{diag}\left(K\left(\frac{\int X_1\beta-\int X_j\beta}{h_n}\right),\cdots,K\left(\frac{\int X_n\beta-\int X_j\beta}{h_n}\right)\right). \nonumber
\end{align}
The estimation of the coefficient function $\beta^0$ and the link function $g$ is therefore a nested procedure, summarized in $\left(\ref{step1}\right)$, $\left(\ref{step2}\right)$ and $\left(\ref{step3}\right)$. Following \cite{ma2016estimation}, the identifiability of the model is ensured by adding a constraint on the coefficient function, such that $\int_0^1 \beta^2\left(t\right)\mathrm{d}t = 1.$

\section{Assumptions}
In deriving a convergence rate for $\frac{1}{n} \sum\limits_{i=1}^n \mathrm{E}\left[\hat{g}''\left(\int X_i\hat{\beta}\right) - g''\left(\int X_i\beta^0\right)\right]^2$, we make the following assumptions in the functional single index model.
\begin{enumerate}
\item The observations $\left(X_i\left(t\right),Y_i\right)$, where $i = 1,\cdots,n$, are independent and identically distributed. Each covariate function $X_i\left(t\right)$ is a square-integrable random function defined in the interval $\left[0,1\right]$. The random error $\epsilon$ is independent from $X$, and has zero mean and variance $\sigma^2$.
\item The dependent variable $Y$ has the $m$th-order absolute moment, where $m \geq 2$. This is an assumption from \cite{ichimura1993semiparametric}. The finite moment $m$ is used in establishing the main convergence theorem.
\item The link function $g$ and the curvature $g''$ are bounded and satisfy the Lipschitz condition such that 
\begin{align}
\left|g^{\left(k\right)}\left(u\right) - g^{\left(k\right)}\left(v\right)\right| \leq D_2\left|u - v\right|, \nonumber
\end{align}
for all $u$ and $v$, where $D_2 > 0$ and $k = 0,2$. The Lipschitz condition ensures that if $\beta^0$ can be estimated root-$n$ consistently, the distance between $g''\left(\int X\beta^0\right)$ and $g''\left(\int X\hat{\beta}\right)$ can be controlled.
\item  The kernel function $K$ is nonnegative and symmetric with support $\left[-1,1\right]$, and $\int_{-1}^1 K\left(s\right)\mathrm{d}s = 1$. Assume that $K$ is three times continuously differentiable, with $\left|K^{\left(3\right)}\left(s\right)\right| \leq D_3$, for any $s \in \left[-1,1\right]$ and $D_3 > 0$. Since the kernel $K$ satisfies a Lipschitz condition, the Nadaraya-Watson estimator of the link function $g$ also has a Lipschitz condition.
\item For some orthonormal basis $\left\{\phi_k\left(t\right): k = 1,2,\cdots\right\}$, for each $i = 1,\cdots,n$, there exists a sequence of random variables $\left\{c_{ij}\right\}_{j=1}^{\infty}$, such that
\begin{align}
X_i\left(t\right) = \sum\limits_{j=1}^{\infty} c_{ij}\phi_j\left(t\right), \nonumber
\end{align}
and
\begin{align}
\beta\left(t\right) = \sum\limits_{j=1}^{\infty} b_j\phi_j\left(t\right). \nonumber
\end{align}
Assume that $\mathrm{E}\left(c_{ij}\right) = 0$.\\
In particular, we have
\begin{align}
\hat{\beta}\left(t\right) = \sum\limits_{j=1}^{\infty} \hat{b}_j\phi_j\left(t\right), \quad \quad  \beta^0\left(t\right) = \sum\limits_{j=1}^{\infty} b^0_j\phi_j\left(t\right).  \nonumber
\end{align}
For any $\beta \in \mathcal{B}$, we can write
\begin{align}
\int X_i\beta = \sum\limits_{j=1}^{\infty} c_{ij}b_j. \nonumber
\end{align}
We observe that an orthonormal basis approximation of the covariate function and coefficient function transforms an integration to an infinite sum. In addition, define a sequence $p_n$ such that $p_n \rightarrow \infty$ as $n \rightarrow \infty$, we require
\begin{align}
\sum\limits_{j=p_n+1}^{\infty} c_{ij}b_j = O\left(p_n^{-\lambda}\right),\label{dim}
\end{align}
where $\lambda > 0$, and $p_n = o\left(\frac{1}{h_n}\right)$. Condition $\left(\ref{dim}\right)$ ensures that the integration $\int X\beta$ can be approximated by a finite sum of coefficients under an orthonormal basis.
\item Assume that $\sup_{\beta \in \mathcal{B};x} f\left(x\left|\beta\right.\right) < \infty$, where $f\left(x\left|\beta\right.\right)$ is the probability density of $\int X\beta$.
\end{enumerate}

\section{Convergence Rates}
By the definition of $\mathbb{X}_{\beta}^{\top}$ and $\mathbb{K}_{\beta}$, we can calculate
\begin{align}
\mathbb{X}_{\beta}^{\top}\mathbb{K}_{\beta} =& \left( \begin{array}{c}
\bm{1}^{\top}\\\
\left(\int \bm{X}\beta - \left(\int X_j\beta\right)\bm{1}\right)^{\top}\\
\left[\frac{\left(\int \bm{X}\beta - \left(\int X_j\beta\right)\bm{1}\right)^2}{2}\right]^{\top} \end{array} \right)\left(\begin{array}{ccc} 
K\left(\frac{\int X_1\beta-\int X_j\beta}{h_n}\right) & \bm{0}^{\top} & 0 \\
\bm{0} & \ddots & \bm{0} \\
0 & \bm{0}^{\top} & K\left(\frac{\int X_n\beta-\int X_j\beta}{h_n}\right)
\end{array}\right) \nonumber \\
=& \left(\begin{array}{c} 
\left[K\left(\frac{\int \bm{X}\beta-\left(\int X_j\beta\right)\bm{1}}{h_n}\right)\right]^{\top} \\
\left[\left(\int \bm{X}\beta-\left(\int X_j\beta\right)\bm{1}\right)K\left(\frac{\int \bm{X}\beta-\left(\int X_j\beta\right)\bm{1}}{h_n}\right)\right]^{\top} \\
\left[\frac{\left(\int \bm{X}\beta-\left(\int X_j\beta\right)\bm{1}\right)^2}{2}K\left(\frac{\int \bm{X}\beta-\left(\int X_j\beta\right)\bm{1}}{h_n}\right)\right]^{\top} 
\end{array}\right). \nonumber
\end{align}
Denote
\begin{align}
T_j^p\left(\beta\right) =& \sum\limits_{i=1}^n \left(\int X_i\beta-\int X_j\beta\right)^pK\left(\frac{\int X_i\beta-\int X_j\beta}{h_n}\right), \nonumber
\end{align}
and
\begin{align}
T_j^0\left(\beta\right) =& \sum\limits_{i=1}^n K\left(\frac{\int X_i\beta-\int X_j\beta}{h_n}\right).\nonumber
\end{align}
Denote $u_j \doteq \int X_j\beta$, we have
\begin{align}
T_j^0\left(\beta\right) =& \sum\limits_{i=1}^n K\left(\frac{\int X_i\beta-\int X_j\beta}{h_n}\right) \nonumber \\
=& n \int K\left(\frac{z-u_j}{h_n}\right)f\left(z\left|\beta\right.\right)\mathrm{d}z + O\left(1\right) \nonumber \\
=& n \int K\left(m\right)f\left(u_j+h_nm\left|\beta\right.\right)h_n\mathrm{d}m + O\left(1\right) \nonumber \\
=& nh_nf\left(u_j\left|\beta\right.\right)\int K\left(m\right)\mathrm{d}m + O\left(h_n^2\right) + O\left(1\right)\nonumber \\
=& nh_nf\left(\int X_j\beta\left|\beta\right.\right) + O\left(1\right), \nonumber 
\end{align}
and 
\begin{align}
T_j^p\left(\beta\right) =& \sum\limits_{i=1}^n \left(\int X_i\beta-\int X_j\beta\right)^pK\left(\frac{\int X_i\beta-\int X_j\beta}{h_n}\right) \nonumber \\
=& n \int \left(z-u_j\right)^pK\left(\frac{z-u_j}{h_n}\right)f\left(z\left|\beta\right.\right)\mathrm{d}z + O\left(1\right) \nonumber \\
=& n\int \left(h_nm\right)^pK\left(m\right)f\left(u_j+h_nm\left|\beta\right.\right)h_n\mathrm{d}m + O\left(1\right)\nonumber \\
=& nh_n^{p+1}f\left(u_j\left|\beta\right.\right)\int m^pK\left(m\right)\mathrm{d}m + O\left(h_n^{p+2}\right) + O\left(1\right) \nonumber  \\
=& nh_n^{p+1}f\left(\int X_j\beta\left|\beta\right.\right)\mu_p\left(K\right) + O\left(1\right), \nonumber 
\end{align}
where $\mu_p\left(K\right) \doteq \int_{-1}^1 m^pK\left(m\right)\mathrm{d}m$. Since the kernel function $K$ is symmetric, $T_j^p\left(\beta\right) = 0$ if $p$ is an odd number.\\
We have
\begin{align}
\mathbb{X}_{\beta,j}^{\top}\mathbb{K}_{\beta,j}\mathbb{X}_{\beta,j} \approx \left(\begin{array}{ccc} 
T_j^0\left(\beta\right) & T_j^1\left(\beta\right) & \frac{T_j^2\left(\beta\right)}{2} \\
T_j^1\left(\beta\right) & T_j^2\left(\beta\right) & \frac{T_j^3\left(\beta\right)}{2} \\
\frac{T_j^2\left(\beta\right)}{2} & \frac{T_j^3\left(\beta\right)}{2} & \frac{T_j^4\left(\beta\right)}{4}
\end{array}\right) = \left(\begin{array}{ccc} 
T_j^0\left(\beta\right) & 0 & \frac{T_j^2\left(\beta\right)}{2} \\
0 & T_j^2\left(\beta\right) & 0 \\
\frac{T_j^2\left(\beta\right)}{2} & 0 & \frac{T_j^4\left(\beta\right)}{4}
\end{array}\right).\nonumber
\end{align}
The determinant of the matrix $\mathbb{X}^{\top}_{\beta,j}\mathbb{K}_{\beta,j}\mathbb{X}_{\beta,j}$ is
\begin{align}
\left|\mathbb{X}^{\top}_{\beta,j}\mathbb{K}_{\beta,j}\mathbb{X}_{\beta,j}\right| = \frac{T_j^0\left(\beta\right)T_j^2\left(\beta\right)T_j^4\left(\beta\right)-\left(T_j^2\left(\beta\right)\right)^3}{4}.\nonumber
\end{align}
By definitions of $S_1\left(\beta;j\right)$ and $S_2\left(\beta;j\right)$ in $\left(\ref{s1}\right)$ and $\left(\ref{s2}\right)$, we can get
\begin{align}
&S_1\left(\beta;j\right) = \frac{\left[T_j^0\left(\beta\right)T_j^4\left(\beta\right)-\left(T_j^2\left(\beta\right)\right)^2\right]\left[\left(\int \bm{X}\beta-\left(\int X_j\beta\right)\bm{1}\right)K\left(\frac{\int \bm{X}\beta-\left(\int X_j\beta\right)\bm{1}}{h_n}\right)\right]^{\top}}{T_j^0\left(\beta\right)T_j^2\left(\beta\right)T_j^4\left(\beta\right)-\left(T_j^2\left(\beta\right)\right)^3},\nonumber \\
&S_2\left(\beta;j\right) \nonumber \\
& \hspace{.4cm} = \frac{2T_j^0\left(\beta\right)\left[\left(\int \bm{X}\beta-\left(\int X_j\beta\right)\bm{1}\right)^2K\left(\frac{\int \bm{X}\beta-\left(\int X_j\beta\right)\bm{1}}{h_n}\right)\right]^{\top}-2T_j^2\left(\beta\right)\left(K\left(\frac{\int \bm{X}\beta-\left(\int X_j\beta\right)\bm{1}}{h_n}\right)\right)^{\top} }{T_j^0\left(\beta\right)T_j^4\left(\beta\right) - \left(T_j^2\left(\beta\right)\right)^2}. \nonumber
\end{align}

\subsection{Convergence Rate of $\beta$}
For any $\beta \in \mathcal{B}$, define
\begin{align}
J_n\left(\beta\right) \doteq \frac{1}{n}\sum\limits_{i=1}^n \left[Y_i-\hat{g}\left(\int X_i\beta\right)\right]^2 = \frac{1}{n}\sum\limits_{i=1}^n \left[Y_i - \hat{g}\left(\sum\limits_{j=1}^{\infty} c_{ij}b_j\right)\right]^2, \nonumber
\end{align}
and 
\begin{align}
J_{n,p_n}\left(\beta\right) \doteq \frac{1}{n}\sum\limits_{i=1}^n \left[Y_i - \hat{g}\left(\sum\limits_{j=1}^{p_n} c_{ij}b_j\right)\right]^2. \nonumber
\end{align}
Denote $\bm{c}_i = \left(c_{i1},\cdots\right)^{\top}$ and $\bm{b} = \left(b_1,\cdots\right)^{\top}$, where $i = 1, \cdots, n$. Define the subspaces $C, B \subset R^{\infty}$ such that $\bm{c}_i \in C$ and $\bm{b} \in B$. 
\begin{lemma}
1 For a sequence of positive numbers $M_n$, suppose that $nh_n^2\epsilon_{0n}^2M_n^{-2} \rightarrow \infty$, then
\begin{align}
\mathrm{P}\left\{\sup_{C \times B} \left|\sum\limits_{i=1}^n \left[\hat{g}_{ni}\left(\int x\beta\right) - \mathrm{E}\left(\hat{g}_{ni}\left(\int x\beta\right)\right)\right]\right| \geq \epsilon_{0n}\right\} \rightarrow 0, \nonumber
\end{align}
where
\begin{align}
\hat{g}_{ni}\left(\int x\beta\right) = \frac{1}{h_n}Y_iI\left(Y_i \in \left[-M_n,M_n\right]\right)K\left(\frac{\int x\beta - \int X_i\beta}{h_n}\right), \nonumber
\end{align}
as $n \rightarrow \infty$.
\end{lemma}
\begin{proof}
We have
\begin{align}
& \mathrm{P}\left\{\sup_{C \times B} \left|\frac{1}{n}\sum\limits_{i=1}^n \left[\hat{g}_{ni}\left(\int x\beta\right) - \mathrm{E}\left(\hat{g}_{ni}\left(\int x\beta\right)\right)\right]\right| \geq \epsilon_{0n}\right\} \nonumber \\
& \hspace{.4cm} = \mathrm{P}\left\{\sup_{C \times B} \left|\sum\limits_{i=1}^n \left[Y_iI\left(Y_i \in \left[-M_n,M_n\right]\right)K\left(\frac{\int x\beta - \int X_i\beta}{h_n}\right)\right.\right.\right. \nonumber \\
& \hspace{.8cm} - \left. \left.\left.\mathrm{E}\left(Y_iI\left(Y_i \in \left[-M_n,M_n\right]\right)K\left(\frac{\int x\beta - \int X_i\beta}{h_n}\right)\right)\right]\right| \geq nh_n\epsilon_{0n}\right\} \nonumber \\
& \hspace{.4cm} \doteq \mathrm{P}\left(\sup_{C \times B} \left|\sum\limits_{i=1}^n A_i\right| \geq nh_n\epsilon_{0n}\right), \nonumber
\end{align}
where
\begin{align}
A_i = Y_iI\left(Y_i \in \left[-M_n,M_n\right]\right)K\left(\frac{\int x\beta - \int X_i\beta}{h_n}\right) - \mathrm{E}\left(Y_iI\left(Y_i \in \left[-M_n,M_n\right]\right)K\left(\frac{\int x\beta - \int X_i\beta}{h_n}\right)\right). \nonumber
\end{align}
By Assumption $4$, the kernel function $K$ is bounded. We will apply Bernstein's inequality (see Appendix A) to the above equation with
\begin{align}
\eta_n = nh_n\epsilon_{0n}, \nonumber
\end{align}
\begin{align}
\left|A_i\right| \leq 2\left|Y_i\right|K_1 \leq 2M_nC_1 \doteq c_n,\nonumber
\end{align}
and 
\begin{align}
\text{var}\left(A_i\right) &\leq M_n^2C_2, \nonumber\\
V_n \doteq nM_n^2C_2 &\geq \sum\limits_{i=1}^n\text{var}\left(A_i\right),\nonumber
\end{align}
where $K_1$, $C_1$ and $C_2$ are constants. By Bernstein's inequality, we can get
\begin{align}
& \mathrm{P}\left\{\sup_{C \times B} \left|\frac{1}{n}\sum\limits_{i=1}^n \left[\hat{g}_{ni}\left(\int x\beta\right) - \mathrm{E}\left(\hat{g}_{ni}\left(\int x\beta\right)\right)\right]\right| \geq \epsilon_{0n}\right\}  \nonumber \\
&\hspace{.4cm} \leq \exp\left[-\frac{\eta_n^2}{2\left(V_n+\frac{1}{3}c_n\eta_n\right)}\right] \nonumber \\
&\hspace{.4cm} =\exp\left[-\frac{\left(nh_n\epsilon_{0n}\right)^2}{2\left(nM_n^2C_2+\frac{2}{3}M_nC_1nh_n\epsilon_{0n}\right)}\right] \nonumber \\
& \hspace{.4cm} =\exp\left[-\frac{nh_n\epsilon_{0n}^2}{\frac{2M_n^2C_2}{h_n}+\frac{4}{3}M_nC_1\epsilon_{0n}}\right]. \nonumber 
\end{align}
The assumption of Lemma $A.5$ in \cite{ichimura1993semiparametric} is that the sequence $\left\{M_n\right\}_{n=1}^{\infty}$ should satisfy $\epsilon_{0n}h_nM_n^{m-1} \rightarrow \infty$. Since $h_n \rightarrow 0$ and $\epsilon_{0n} \rightarrow 0$, we need $M_n \rightarrow \infty$. Therefore, in the denominator, we have $\frac{4}{3}M_nK_1\epsilon_{0n} = o\left(\frac{2M_n^2K_2}{h_n}\right)$. If $nh_n^2\epsilon_{0n}^2M_n^{-2} \rightarrow \infty$, then
\begin{align}
-\frac{nh_n\epsilon_{0n}^2}{\frac{2M_n^2K_2}{h_n}+\frac{4}{3}M_nK_1\epsilon_{0n}} \rightarrow -\infty, \nonumber
\end{align}
and
\begin{align}
\mathrm{P}\left\{\sup_{C \times B} \left|\frac{1}{n}\sum\limits_{i=1}^n \left[\hat{g}_{ni}\left(\int x\beta\right) - \mathrm{E}\left(\hat{g}_{ni}\left(\int x\beta\right)\right)\right]\right| \geq \epsilon_{0n}\right\} \rightarrow 0.\nonumber
\end{align}
\end{proof}
Following similar arguments, we can derive Lemmas $2$ and $3$ below.
\begin{lemma}
2 For a sequence of positive numbers $M_n$, suppose that $nh_n^4\epsilon_{1n}^2M_n^{-2} \rightarrow \infty$, then
\begin{align}
\mathrm{P}\left\{\sup_{C \times B} \left|\sum\limits_{i=1}^n \left[\hat{g}'_{ni}\left(\int x\beta\right) - \mathrm{E}\left(\hat{g}'_{ni}\left(\int x\beta\right)\right)\right]\right| \geq \epsilon_{1n}\right\} \rightarrow 0, \nonumber
\end{align}
where
\begin{align}
\hat{g}'_{ni}\left(\int x\beta\right) =& \left[\frac{1}{h_n}Y_iI\left(Y_i \in \left[-M_n,M_n\right]\right)K\left(\frac{\int x\beta - \int X_i\beta}{h_n}\right)\right]' \nonumber\\
\doteq &  \frac{1}{h_n^2}Y_iI\left(Y_i \in \left[-M_n,M_n\right]\right)H_1\left(x,X_i,\beta\right)K'\left(\frac{\int x\beta - \int X_i\beta}{h_n}\right), \nonumber
\end{align}
as $n \rightarrow \infty$.
\end{lemma}
\begin{lemma}
3 For a sequence of positive numbers $M_n$, suppose that $nh_n^6\epsilon_{2n}^2M_n^{-2} \rightarrow \infty$, then
\begin{align}
\mathrm{P}\left\{\sup_{C \times B} \left|\sum\limits_{i=1}^n \left[\hat{g}''_{ni}\left(\int x\beta\right) - \mathrm{E}\left(\hat{g}''_{ni}\left(\int x\beta\right)\right)\right]\right| \geq \epsilon_{2n}\right\} \rightarrow 0,\nonumber
\end{align}
where
\begin{align}
\hat{g}''_{ni}\left(\int x\beta\right) =& \left[\frac{1}{h_n}Y_iI\left(Y_i \in \left[-M_n,M_n\right]\right)K\left(\frac{\int x\beta - \int X_i\beta}{h_n}\right)\right]'' \nonumber \\
\doteq &  \frac{1}{h_n^3}Y_iI\left(Y_i \in \left[-M_n,M_n\right]\right)H_2\left(x,X_i,\beta\right)K''\left(\frac{\int x\beta - \int X_i\beta}{h_n}\right), \nonumber
\end{align}
as $n \rightarrow \infty$.
\end{lemma}
Now, we show the root-$n$ consistency of the estimator $\hat{\beta}$.
\begin{theorem}
4 The estimator $\hat{\beta}$ is consistent if Assumptions $1-6$ hold, and the bandwidth sequence satisfies $nh_n^8 \rightarrow 0$ and $nh_n^6 \rightarrow \infty$.
\end{theorem}
\begin{proof}
Theorem $5.1$ in \cite{ichimura1993semiparametric} states the consistency of $\hat{\beta}$, while $\beta^0$ is a coefficient vector. Theorem $5.1$ is based on Lemma $5.1$, and Lemma $5.1$ is based on Lemmas $A.2 - A.10$. The proof of Lemmas $A.2 - A.7$ will be the same whenever $\beta^0$ is a vector or a function. Lemmas $1 - 3$ above are the functional version of Lemmas $A.8 - A.10$. We need to figure out the constraint on the smoothing parameter $h_n$. The constraints are
\begin{itemize}
\item Lemma $A.2 - A.4$ 
\begin{align}
nh_n^8 \rightarrow 0.\nonumber
\end{align}
\item Lemma $A.5 - A.7$
\begin{align}
\epsilon_{0n}h_nM_n^{m-1} \rightarrow \infty, \quad \epsilon_{1n}h_n^2M_n^{m-1} \rightarrow \infty,\quad \epsilon_{2n}h_n^3M_n^{m-1} \rightarrow \infty.\nonumber
\end{align}
\item Lemma $A.8 - A.10$
\begin{align}
n\epsilon_{0n}^2h_n^2M_n^{-2} \rightarrow \infty, \quad n\epsilon_{1n}^2h_n^4M_n^{-2} \rightarrow \infty,\quad n\epsilon_{2n}^2h_n^6M_n^{-2} \rightarrow \infty.\nonumber
\end{align}
\end{itemize}
We require $nh_n^2\epsilon_{0n}^2M_n^{-2} \rightarrow \infty$. Since $\epsilon_{0n} \rightarrow 0$ and $M_n^{-2} \rightarrow 0$, we need to have $nh_n^2 \rightarrow \infty$. Following the same argument, we need $nh_n^4 \rightarrow \infty$ and $nh_n^6 \rightarrow \infty$.
\end{proof}
To prove the convergence rate of the functional single index model, we need to find a convergence rate for $\hat{\beta}$, $\hat{\bm{b}}$ or $\int X\hat{\beta}$.
\begin{theorem}
5 Suppose that $nh_n^6 \rightarrow \infty$, $nh_n^8 \rightarrow 0$ and $\frac{nh_n^{3+\frac{3}{m-1}}}{-\log h_n} \rightarrow \infty$, then we have
\begin{align}
\sqrt{n}\left(\int X_{i}\hat{\beta} - \int X_{i}\beta^0\right) = O_p\left(1\right),\nonumber
\end{align}
for $i = 1,\cdots,n$.
\end{theorem}
\begin{proof}
For each $p_n$ such that $p_n \rightarrow \infty$ as $n \rightarrow \infty$, define
\begin{align}
\hat{\bm{b}}_{p_n} = \left(\hat{b}_1,\cdots,\hat{b}_{p_n}\right), \text{ and, }\bm{b}_{p_n}^0 = \left(b_1^0,\cdots,b_{p_n}^0\right).\nonumber
\end{align}
Since the kernel function $K$ satisfies the Lipschitz condition, and the estimated link function $g$ is
\begin{align}
\hat{g}\left(\int X_i\beta\right) = \frac{\sum\limits_{j:j \neq i} Y_jK\left(\frac{\int X_i\beta - \int X_j\beta}{h_n}\right)}{\sum\limits_{j:j \neq i} K\left(\frac{\int X_i\beta - \int X_j\beta}{h_n}\right)},\nonumber
\end{align}
the estimated $\hat{g}$ also satisfies the Lipschitz condition. Since $\sum\limits_{j=p_n+1}^{\infty} c_{ij}b_j = O\left(p_n^{-\lambda}\right)$, we have
\begin{align}
\left|\hat{g}\left(\sum\limits_{j=1}^{\infty} c_{ij}b_j\right) - \hat{g}\left(\sum\limits_{j=1}^{p_n} c_{ij}b_j\right)\right| \leq K_2\left|\sum\limits_{j=p_n+1}^{\infty} c_{ij}b_j\right| = O\left(p_n^{-\lambda}\right), \label{ye}
\end{align}
as $n \rightarrow \infty$, where $K_2$ is a constant. Therefore, $\hat{g}\left(\sum\limits_{j=1}^{p_n} c_{ij}b_j\right) \rightarrow \hat{g}\left(\sum\limits_{j=1}^{\infty} c_{ij}b_j\right)$. By the definition of $J_{n,p_n}\left(\beta\right)$, we can get $J_{n,p_n}\left(\beta\right) \rightarrow J_n\left(\beta\right)$ as $n \rightarrow \infty$, for any $\beta \in \mathcal{B}$. \\
Hence, by Lemma $5.4$ in \cite{ichimura1993semiparametric}, we have
\begin{align}
\sqrt{n}\left(\hat{\bm{b}}_{p_n} - \bm{b}^0_{p_n}\right) \rightarrow D \doteq \mathrm{N}\left(\bm{0}_{p_n}, \mathbb{V}^{-1}\mathbb{\Sigma}\mathbb{V}^{-1}\right),\nonumber
\end{align}
where $\bm{0}_{p_n}$ is a $p_n$-dimensional mean vector, and $\mathbb{V}^{-1}\mathbb{\Sigma}\mathbb{V}^{-1}$ is a $\left(p_n \times p_n\right)$-dimensional covariance matrix. Suppose that $\mathcal{X}$ is the $\sigma$-algebra generated by $\left(X_1,\cdots,X_n\right)$, the $\left(k,m\right)$th-term of the matrix $\mathbb{V}$ is 
\begin{align}
\mathbb{V}_{km} =& \mathrm{E}\left[\frac{\partial^2 J_{n,p_n}\left(\beta^0\right)}{\partial b^0_k \partial b^0_m}\left|\mathcal{X}\right.\right] \nonumber\\
=& \mathrm{E}\left\{\frac{2}{n}\sum\limits_{i=1}^n\left[\hat{g}'\left(\sum\limits_{j=1}^{p_n} c_{ij}b^0_j\right)\right]^2c_{ik}c_{im} - \frac{2}{n}\sum\limits_{i=1}^n\left[Y_i - \hat{g}\left(\sum\limits_{j=1}^{p_n} c_{ij}b^0_j\right)\right]\hat{g}''\left(\sum\limits_{j=1}^{p_n} c_{ij}b^0_j\right)c_{ik}c_{im}\left|\mathcal{X}\right.\right\}  \nonumber \\
=& \mathrm{E}\left\{\frac{2}{n}\sum\limits_{i=1}^n\left[\hat{g}'\left(\sum\limits_{j=1}^{p_n} c_{ij}b^0_j\right)\right]^2c_{ik}c_{im}\left|\mathcal{X}\right.\right\} - \mathrm{E}\left\{\frac{2}{n}\sum\limits_{i=1}^n\left[Y_i - \hat{g}\left(\sum\limits_{j=1}^{p_n} c_{ij}b^0_j\right)\right]\hat{g}''\left(\sum\limits_{j=1}^{p_n} c_{ij}b^0_j\right)c_{ik}c_{im}\left|\mathcal{X}\right.\right\}.  \nonumber
\end{align}
For any $i = 1,\cdots,n$, we have
\begin{align}
\mathrm{E}\left(Y_i\left|\mathcal{X}\right.\right) = g\left(\int X_i\beta^0\right) = \mathrm{E}\left[\hat{g}\left(\int X_i\beta^0\right)\left|\mathcal{X}\right.\right] + O\left(h_n^2\right), \label{kk}
\end{align}
where the second equality is the bias property of the kernel density estimate. We can calculate
\begin{align}
& \mathrm{E}\left\{\frac{1}{n}\sum\limits_{i=1}^n\left[Y_i - \hat{g}\left(\sum\limits_{j=1}^{p_n} c_{ij}b^0_j\right)\right]\left|\mathcal{X}\right.\right\} \nonumber\\
& \hspace{.4cm} =  \mathrm{E}\left\{\frac{1}{n}\sum\limits_{i=1}^n\left[Y_i - \hat{g}\left(\int X_i\beta^0\right)\right]\left|\mathcal{X}\right.\right\} +  \mathrm{E}\left\{\frac{1}{n}\sum\limits_{i=1}^n\left[\hat{g}\left(\int X_i\beta^0\right) - \hat{g}\left(\sum\limits_{j=1}^{p_n} c_{ij}b^0_j\right)\right]\left|\mathcal{X}\right.\right\}, \nonumber 
\end{align}
where the first term converges to $0$ by the equation $\left(\ref{kk}\right)$, since $h_n \rightarrow 0$ as $n \rightarrow \infty$, and the second term converges to $0$ by $\left(\ref{ye}\right)$. Therefore, we have
\begin{align}
\mathrm{E}\left\{\frac{1}{n}\sum\limits_{i=1}^n\left[Y_i - \hat{g}\left(\sum\limits_{j=1}^{p_n} c_{ij}b^0_j\right)\right]\left|\mathcal{X}\right.\right\} \rightarrow 0,\nonumber
\end{align}
as $n \rightarrow \infty$. By the Slutsky's Theorem, we can get
\begin{align}
\mathrm{E}\left\{\frac{2}{n}\sum\limits_{i=1}^n\left[Y_i - \hat{g}\left(\sum\limits_{j=1}^{p_n} c_{ij}b^0_j\right)\right]\hat{g}''\left(\sum\limits_{j=1}^{p_n} c_{ij}b^0_j\right)c_{ik}c_{im}\left|\mathcal{X}\right.\right\} \rightarrow 0.\nonumber
\end{align}
Therefore,
\begin{align}
\mathbb{V}_{km} =& \mathrm{E}\left\{\frac{2}{n}\sum\limits_{i=1}^n\left[\hat{g}'\left(\sum\limits_{j=1}^{p_n} c_{ij}b^0_j\right)\right]^2c_{ik}c_{im}\left|\mathcal{X}\right.\right\}. \nonumber
\end{align}
For any $\beta \in \mathcal{B}$ and any $j \in \left\{1,\cdots,n\right\}$, we have
\begin{align}
\left|\hat{g}'\left(\int X_j\beta\right)\right| =& \left|S_1\left(\beta;j\right)\bm{Y}\right| \nonumber \\
=& \left|\frac{\left[T_j^0\left(\beta\right)T_j^4\left(\beta\right)-\left(T_j^2\left(\beta\right)\right)^2\right]\sum\limits_{i=1}^n\left(\int X_i\beta-\int X_j\beta\right)K\left(\frac{\int X_i\beta-\int X_j\beta}{h_n}\right)Y_i}{T_j^0\left(\beta\right)T_j^2\left(\beta\right)T_j^4\left(\beta\right)-\left(T_j^2\left(\beta\right)\right)^3}\right| \nonumber \\
\leq & \left|\frac{\left[T_j^0\left(\beta\right)T_j^4\left(\beta\right)-\left(T_j^2\left(\beta\right)\right)^2\right]T_j^1\left(\beta\right)\left|\max_{i \in \left\{1,\cdots,n\right\}}Y_i\right|}{T_j^0\left(\hat{\beta}\right)T_j^2\left(\beta\right)T_j^4\left(\beta\right)-\left(T_j^2\left(\beta\right)\right)^3}\right| \nonumber \\
\approx & \left|\frac{\left[nh_nf\left(u\left|\beta\right.\right)nh_n^5f\left(u\left|\beta\right.\right)\mu_4\left(K\right)-\left(nh_n^3f\left(u\left|\beta\right.\right)\mu_2\left(K\right)\right)^2\right]nh_n^2f\left(u\left|\beta\right.\right)\mu_1\left(K\right)\left|\max_{i \in \left\{1,\cdots,n\right\}}Y_i\right|}{nh_nf\left(u\left|\beta\right.\right)nh_n^3f\left(u\left|\beta\right.\right)\mu_2\left(K\right)nh_n^5f\left(u\left|\beta\right.\right)\mu_4\left(K\right)-\left(nh_n^3f\left(u\left|\beta\right.\right)\mu_2\left(K\right)\right)^3}\right| \nonumber \\
\sim & \frac{1}{h_n}, \nonumber
\end{align}
where $u = \int X_j\beta$.\\
For any $i = 1,\cdots,n$, denote $c_{\cdot \cdot} = c_{ik}$, for any $k \in \left\{1,\cdots,n\right\}$, and $c_{\cdot k} = c_{ik}$ for a fix $k$. For any $k,m \in \left\{1,\cdots,n\right\}$, we have
\begin{align}
\mathbb{V}_{km} \sim \frac{\mathrm{E}\left(c_{ik}c_{im}\left|\mathcal{X}\right.\right)}{h_n^2} \sim \frac{\text{cov}\left(c_{ik},c_{im}\right)}{h_n^2} \sim \frac{\text{cov}\left(c_{\cdot k},c_{\cdot m}\right)}{h_n^2},\nonumber
\end{align}
since $X_i$ are independent and identically distributed with $\mathrm{E}\left(c_{ij}\right) = 0$, for any $i=1,\cdots,n$ and $j=1,\cdots$. \\
The $\left(k,m\right)$th-term of the matrix $\mathbb{\Sigma}$ is
\begin{align}
\mathbb{\Sigma}_{km} =& \mathrm{E}\left[\sigma \frac{\partial J_{n,p_n}\left(\beta^0\right)}{\partial b^0_k}\cdot \sigma\frac{\partial J_{n,p_n}\left(\beta^0\right)}{\partial b^0_m}\right] \nonumber\\
=& \mathrm{E}\left[\frac{2\sigma}{n}\sum\limits_{i=1}^n\left[Y_i - \hat{g}\left(\sum\limits_{j=1}^{p_n} c_{ij}b^0_j\right)\right]\hat{g}'\left(\sum\limits_{j=1}^{p_n} c_{ij}b^0_j\right)c_{ik} \cdot \frac{2\sigma}{n}\sum\limits_{i=1}^n\left[Y_i - \hat{g}\left(\sum\limits_{j=1}^{p_n} c_{ij}b^0_j\right)\right]\hat{g}'\left(\sum\limits_{j=1}^{p_n} c_{ij}b^0_j\right)c_{im}\right] \nonumber \\
\sim & \frac{\text{cov}\left(c_{\cdot k},c_{\cdot m}\right)}{h_n^2}. \nonumber 
\end{align}
The diagonal term of the covariance matrix $\mathbb{V}^{-1}\mathbb{\Sigma}\mathbb{V}^{-1}$ is 
\begin{align}
\left(\mathbb{V}^{-1}\mathbb{\Sigma}\mathbb{V}^{-1}\right)_{kk} \sim \frac{p_n^2h_n^2}{\text{var}\left(c_{\cdot\cdot}\right)}.\nonumber
\end{align}
Define the truncated version of $X_i\left(t\right)$ and $\beta\left(t\right)$ as
\begin{align}
X_{i,p_n}\left(t\right) = \sum\limits_{j=1}^{p_n} c_{ij}\phi_j\left(t\right), \nonumber\\
\beta_{p_n}\left(t\right) = \sum\limits_{j=1}^{p_n} b_j\phi_j\left(t\right),\nonumber
\end{align}
where $i = 1, \cdots, n$. \\
We have for any $i = 1, \cdots, n$, define $\bm{c}_{i,p_n} = \left(c_{i1},\cdots,c_{ip_n}\right)$,
\begin{align}
\sqrt{n}\left(\int X_{i,p_n}\hat{\beta}_{p_n} - \int X_{i,p_n}\beta^0_{p_n}\right) = \sqrt{n}\left(\bm{c}_{i,p_n}\hat{\bm{b}}_{p_n} - \bm{c}_{i,p_n}\bm{b}^0_{p_n}\right) = \bm{c}_{i,p_n}  \sqrt{n}\left(\hat{\bm{b}}_{p_n} - \bm{b}^0_{p_n}\right).\nonumber
\end{align}
Therefore, $\sqrt{n}\left(\int X_{i,p_n}\hat{\beta}_{p_n} - \int X_{i,p_n}\beta^0_{p_n}\right)$ converges to a normal distribution with the covariance matrix 
$\text{var}\left(c_{\cdot\cdot}\right)\cdot \frac{p_n^2h_n^2}{\text{var}\left(c_{\cdot\cdot}\right)} = p_n^2h_n^2$. By Assumption $5$, since $p_n = o_p\left(\frac{1}{h_n}\right)$, we can get $\text{var}\left(c_{\cdot\cdot}\right)\cdot \frac{p_n^2h_n^2}{\text{var}\left(c_{\cdot\cdot}\right)} = o_p\left(1\right)$. When $n \rightarrow \infty$, $p_n \rightarrow \infty$, we have
\begin{align}
\sqrt{n}\left(\int X_{i}\hat{\beta} - \int X_{i}\beta^0\right) = O_p\left(1\right),\nonumber
\end{align}
for $i = 1,\cdots,n$.
\end{proof}

\subsection{Main Theorem}
\begin{theorem}
6 If $nh_n^6 \rightarrow \infty$, $nh_n^8 \rightarrow 0$ and $\frac{nh_n^{3+\frac{3}{m-1}}}{-\log h_n} \rightarrow \infty$, we have
\begin{align}
\frac{1}{n}\sum\limits_{i=1}^n\mathrm{E}\left[\hat{g}''\left(\int X_i\hat{\beta}\right) - g''\left(\int X_i\beta^0\right)\right]^2 = O\left(h_n^4 + \frac{1}{nh_n^4}\right).\nonumber
\end{align}
\end{theorem}
\begin{proof}
Since $X_i$ are independent and identically distributed, for any $j \in \left\{1,\cdots,n\right\}$, we only need to find the convergence rate of $\mathrm{E}\left[\hat{g}''\left(\int X_j\hat{\beta}\right) - g''\left(\int X_j\beta^0\right)\right]^2$. We can decompose it into three terms:
\begin{align}
&\mathrm{E}\left[\hat{g}''\left(\int X_j\hat{\beta}\right) - g''\left(\int X_j\beta^0\right)\right]^2 \nonumber \\
& \hspace{.4cm}\leq  \mathrm{E}\left[\hat{g}''\left(\int X_j\hat{\beta}\right) - \bar{g}''\left(\int X_j\hat{\beta}\right)\right]^2 + \mathrm{E}\left[\bar{g}''\left(\int X_j\hat{\beta}\right) - g''\left(\int X_j\hat{\beta}\right)\right]^2 \nonumber \\
& \hspace{.8cm} + \mathrm{E}\left[g''\left(\int X_j\hat{\beta}\right) - g''\left(\int X_j\beta^0\right)\right]^2. \nonumber 
\end{align}
where 
\begin{align}
\bar{g}''\left(\int X_j\hat{\beta}\right) \doteq S_2\left(\hat{\beta};j\right)\bm{g},\nonumber
\end{align}
and
\begin{align}
\bm{g} \doteq \left(g\left(\int X_1\beta^0\right),\cdots,g\left(\int X_n\beta^0\right)\right)^{\top}.\nonumber
\end{align}
By Lemma $7$, $8$ and $9$ in the Appendix, we have
\begin{align}
&\mathrm{E}\left[\hat{g}''\left(\int X_j\hat{\beta}\right) - \bar{g}''\left(\int X_j\hat{\beta}\right)\right]^2  = O\left(\frac{1}{nh_n^4}\right).\nonumber\\
&\mathrm{E}\left[\bar{g}''\left(\int X_j\hat{\beta}\right) - g''\left(\int X_j\hat{\beta}\right)\right]^2 = O\left(h_n^4 + \frac{1}{nh_n^4}\right).\nonumber\\
&\mathrm{E}\left[g''\left(\int X_j\hat{\beta}\right) - g''\left(\int X_j\beta^0\right)\right]^2 = O\left(\frac{1}{n}\right).\nonumber
\end{align}
Combining these three terms, we obtain
\begin{align}
\frac{1}{n}\sum\limits_{i=1}^n\mathrm{E}\left[\hat{g}''\left(\int X_i\hat{\beta}\right) - g''\left(\int X_i\beta^0\right)\right]^2 = O\left(h_n^4 + \frac{1}{nh_n^4}\right). \nonumber
\end{align}
\end{proof}

\section{Practical Implementation}
\subsection{Initialization} 
In $\left(\ref{step1}\right)$, the coefficient function $\beta^0$ is estimated by minimizing the mean square error of $g$, which is a nonlinear optimization problem. The coefficient function is approximated by a $K$-dimensional Fourier basis, as
\begin{align}
\beta^0\left(t\right) = \bm{\psi}^{\top}\left(t\right)\bm{c},\nonumber
\end{align}
where $\bm{c}$ is a $K$-dimensional column vector, and $t \in \left[0,1\right]$. In order to ensure identifiability, we constraint the coefficient function $\beta^0$ to have $\int \beta^0 = 0$ and $\|\beta^0\|_2 = 1$. Since the Fourier basis is an orthonormal basis, a constraint on $\beta^0$ is equivalent to a constraint on the coefficient vector, such that $\left\|\bm{c}\right\|_2 = 1$. The first constraint can be enforced by dropping any constant terms from the Fourier basis. To compensate for rescaling $\bm{c}$, we also rescale the bandwidth $h$ such that $h = h\left\|\bm{c}\right\|$ in the optimization step. \\
We use {\tt R} function {\tt optim} to minimize the MSE, and an initial value of the coefficient vector $\bm{c}$ is needed, denoted as $\bm{c}_{\text{init}}$. We use three different methods to select the initial value.
\begin{enumerate}
\item Assume each item in $\bm{c}_{\text{init}}$ is equal, such that $\bm{c}_{\text{init}} = \left(\frac{1}{\sqrt{K}},\cdots,\frac{1}{\sqrt{K}}\right)$. We don't have any previous knowledge about the start point, so for simplicity, choose a vector while all items are equal to each other.  
\item Assume $g\left(s\right) = s$, the coefficient vector $\bm{c}_{\text{init}}$ is estimated by minimizing the ordinary least square, and normalize it to be $\left\|\bm{c}_{\text{init}}\right\|_2 = 1$. To obtain an initial for $\beta^0$, we need to specify a structure for the link function $g$. A linear structure of $g$ is obviously the simplest, and could be calculated easily.
\item Generate $1000$ different standard normal distribution random initial vectors (with length $\bm{c}$), and select the best $10$ initial vectors by penalized mean squared error, when the bandwidth is equal to the mean of bandwidth sequence. For each bandwidth $h$, select the initial that minimizes the penalized MSE. That means that in the cross-validation step, we select different initial for different bandwidth $h$. The selected initial is determined by the bandwidth and corresponding penalized MSE.
\end{enumerate}

\subsection{Cross-validation}
We need to select the bandwidth $h_n$ in the kernel density estimation of $g$ and $g''$. We examine two cross-validation methods:
\begin{enumerate}
\item $10$-fold cross-validation. We partition the dataset into $10$ subsamples. Each time, we use $9$ subsamples as the training set and the remaining subsample as the validation set. We will observe that the $10$-fold cross validation method produces similar results with the GCV method introduced below.
\item Fix $j \in \left\{1,\cdots,n\right\}$, define
\begin{align}
\hat{g}\left(\int X_j\beta\right) = \left(\mathbb{X}_{\beta,j}^{\top}\mathbb{K}_{\beta,j}\mathbb{X}_{\beta,j}\right)^{-1}_{1\cdot}\left(\mathbb{X}_{\beta,j}^{\top}\mathbb{K}_{\beta,j}\right)\bm{Y} \doteq S_0\left(\beta;j\right)\bm{Y},\nonumber
\end{align}
where $S_0\left(\beta;j\right)$ is a $n$-dimensional row vector. Denote a $\left(n \times n\right)$-dimensional smoother matrix $\mathbb{S}_{h_n} = \left(S_0\left(\hat{\beta};1\right),\cdots, S_0\left(\hat{\beta};n\right)\right)^{\top}$. We can get $\hat{\bm{Y}} = \mathbb{S}_{h_n}\bm{Y}$, where $\hat{\bm{Y}} \doteq \left(\hat{g}\left(\int X_1\hat{\beta}\right),\cdots,\hat{g}\left(\int X_n\hat{\beta}\right)\right)^{\top}$. The generalized cross-validation criterion is 
\begin{align}
\text{GCV}\left(h_n\right) \doteq \frac{\frac{1}{n}\left\|\left(\mathbb{I}-\mathbb{S}_{h_n}\right)\bm{Y}\right\|_2^2}
{\left[\frac{1}{n}\text{tr}\left(\mathbb{I}-\mathbb{S}_{h_n}\right)\right]^2},\nonumber
\end{align}
where $\mathbb{I}$ is a $n$-dimensional identity matrix. Find a bandwidth $h_n$ minimizing the $\text{GCV}\left(h_n\right)$. Note that this does not account for estimating $\beta^0$.
\end{enumerate}

\subsection{Simulation Study}
In order to obtain simulated functional data, we defined a $25$-dimensional Fourier basis $\bm{\psi}\left(t\right)$, where $t \in \left[0,1\right]$. The covariate function $X\left(t\right)$ is defined in $\left[0,1\right]$, such that
\begin{align}
X\left(t\right) = \sum\limits_{i=1}^{25} \eta_i\psi_i\left(t\right),\nonumber
\end{align}
where $\eta_i = \frac{i-1}{24}\cdot \mathrm{N}\left(0,1\right)$. The coefficient function is 
\begin{align}
\beta^0\left(t\right) = \bm{a}^{\top}\bm{\psi}\left(t\right),\nonumber
\end{align}
where $\bm{a} = \left(0,1,1,0.5,0,\cdots,0\right)^{\top}$. We use three different link functions:
\begin{enumerate}
\item $g\left(s\right) = e^{-s}$.
\item $g\left(s\right) = -s^2$.
\item $g\left(s\right) = s$.
\end{enumerate}
In order to measure the performance of our estimators, we define the MSE of the estimated $\beta^0$ and $g^{\left(k\right)}$ to be
\begin{align}
\text{RSE} = \left[\int_0^1 \left(\hat{\beta}\left(t\right) - \beta^0\left(t\right)\right)^2\mathrm{d}t\right]^{\frac{1}{2}},\nonumber
\end{align}
and
\begin{align}
\text{RASE(k)} = \left\{\frac{1}{n}\sum\limits_{i=1}^n\left[\hat{Y}^{\left(k\right)}_i - g^{\left(k\right)}\left(\int_0^1 X_i\left(t\right)\beta^0\left(t\right)\mathrm{d}t\right)\right]^2\right\}^{\frac{1}{2}},\nonumber
\end{align}
where $\hat{Y}^{\left(k\right)}_i = \hat{g}^{\left(k\right)}\left(\int_0^1 X_i\left(t\right)\hat{\beta}\left(t\right)\mathrm{d}t\right)$.\\
Theorem $6$ shows that for consistency, the bandwidth must scale between $O\left(n^{-\frac{1}{6}}\right)$ and $O\left(n^{-\frac{1}{8}}\right)$. Assume that the optimal bandwidth for the curvature $g''$ is $n^{-\frac{1}{7}}$, by \cite{chen2011single}, we can expect that the optimal bandwidth for the link function $g$ is $n^{-\frac{1}{5}}$. We use either the GCV or the $10$-fold cross-validation to select a bandwidth h, and rescale it to be $h^{\frac{5}{7}}$. As we have discussed before, we constrain the coefficient $\left\|\beta^0\right\|_2 = 1$; this was achieved by rescaling the bandwidth in our objective function. After solving the nonlinear optimization problem, we rescale both $\bm{c}$ and the bandwidth. The following table shows the simulation results of the GCV method. The number of data points we use are $n = 100$ and $n = 1000$. In Table $2$, we also show the RASE$2$ results without rescaling the bandwidth $h$ to be $h\left\|\bm{c}\right\|$ and $h^{\frac{5}{7}}$, if we start from random initials. We can conclude that re-scaling does matter to the final results and it reduces the error of RASE$2$. The $10$-fold cross-validation results in the Appendix C also confirm that by re-scaling the bandwidth, we improve our estimate for $g''$.
\begin{table}[H]
\caption{Simulation results by GCV using rescaled bandwidth $h\left\|\bm{c}\right\|$ and constraint $\hat{\beta}$.} \label{table8}
\begin{center}
\begin{tabular}{ccccc|ccc|cccc}
\hline
\multirow{2}{*}{}&\multirow{2}{*}{} & \multicolumn{3}{c}{\textbf{g1}} &  \multicolumn{3}{c}{\textbf{g2}} & \multicolumn{3}{c}{\textbf{g3}} \\
\cline{3-11}
\textbf{Initial} & \textbf{n} & \textbf{RSE} & \textbf{RASE} & \textbf{RASE2} & \textbf{RSE} & \textbf{RASE} & \textbf{RASE2} & \textbf{RSE} & \textbf{RASE} & \textbf{RASE2} & \\
\hline
\multirow{2}{*}{True} & 100 & 0.2887 & 0.0959 & 9.4156 & 0.2785 & 0.0945 & 7.2648 & 0.3442 & 0.0918 & 5.0951\\
 & 1000& 0.0829 & 0.0336 & 0.9003 & 0.0731 & 0.0334 & 0.5429 & 0.1095 & 0.0276 & 0.8620 \\
\hline
\multirow{2}{*}{Linear} & 100 & 1.1401 & 0.3401 &1.3869 & 1.3569 & 0.4166 & 1.2014 & 0.7049 & 0.1393 & 0.2507\\
& 1000 & 1.9010 & 0.2340 & 0.5205 & 1.0719 & 0.3421 & 1.3075 & 0.5287 & 0.0891 & 0.0729 \\
\hline
\multirow{2}{*}{Equal} & 100 & 0.8389 & 0.2820 & 1.0933 & 0.8342 & 0.2566 & 1.4349 & 0.8329 & 0.1857 & 0.2589 \\
& 1000 & 0.8502 & 0.3159 & 1.0133 & 0.8437 & 0.2838 & 1.4043 & 0.8508 & 0.1921 & 0.0905 \\
\hline
\multirow{2}{*}{Random} & 100 & 1.1911 & 0.1461 & 0.5597 & 1.1617 & 0.1979 & 0.6186 & 1.2603 & 0.0942 & 0.2043\\
& 1000 & 1.2704 & 0.1320 & 0.4354 & 1.3021 & 0.1890 & 0.4413 & 1.2300 & 0.0796 & 0.0637 
\\[1ex] \hline
\end{tabular}
\end{center}
\end{table}
\noindent We observe that RSE and RASE, as expected, achieve best performance when we initialize our optimizer at the true values. However, the more natural Linear initialization strategy does not outperform initializing at Equal coefficients. For both RASE and RASE$2$, a more intensive search over initializations pays off; in the case of RASE$2$ this even outperforms starting from true values. We suspect that this is associated with differing optimal smoothness criteria. Table $2$ compares our results to when we do not use the $h^{\frac{5}{7}}$ re-scaling where applying this has a significant effect;  by starting from a position far from the optimum, a large bandwidth may have the effect of smoothing the objective function, indirectly improving our re-estimate of $g"$.  \\
Overall, the random initial produces a much better results compared to other initial strategies. Starting from $1000$ initial vectors, we have a chance to select the best $10$ initial vectors, which increases the probability of selecting a ``good" starting points and decreases the chance of converging to a local minimum for the non-linear optimization problem, although this comes at a significant computational prices. In addition, we do observe an improvement of RASE and RASE$2$ when $n = 1000$ compared to $n = 100$.
\begin {table}[H]
\caption {Simulation results of random starting value with rescaled and original bandwidths selected by GCV} \label{table5} 
\begin{center}
\begin{tabular}{ccc|cc|ccc}
\hline
\multirow{2}{*}{} & \multicolumn{2}{c}{\textbf{g1}} &  \multicolumn{2}{c}{\textbf{g2}} & \multicolumn{2}{c}{\textbf{g3}} \\
\cline{2-7}
\textbf{n} & \textbf{original} & \textbf{rescaled}  & \textbf{original} & \textbf{rescaled}  &  \textbf{original} & \textbf{rescaled} \\
\hline
\multirow{2}{*}{}  100 & 1.2793  &0.5597 & 2.1497 & 0.6186  & 0.2360 &  0.2043\\ 
 1000 & 1.1297 & 0.4354   & 2.2081 & 0.4413 & 0.0703 & 0.0637
\\[1ex] \hline
\end{tabular}
\end{center}
\end{table}
\noindent In order to provide a visual sense of the performance of our estimate, in Figures $1$ and $2$, we plot the estimates of the link function $g\left(s\right) = e^{-s}$ and $g\left(s\right) = s$, respectively. We observe that the estimate and the true curve of the link function $g$ almost overlap with each other, but the second derivative has significantly larger error relative to the truth. 
\begin{figure}
\begin{center}
  \includegraphics[width=15cm]{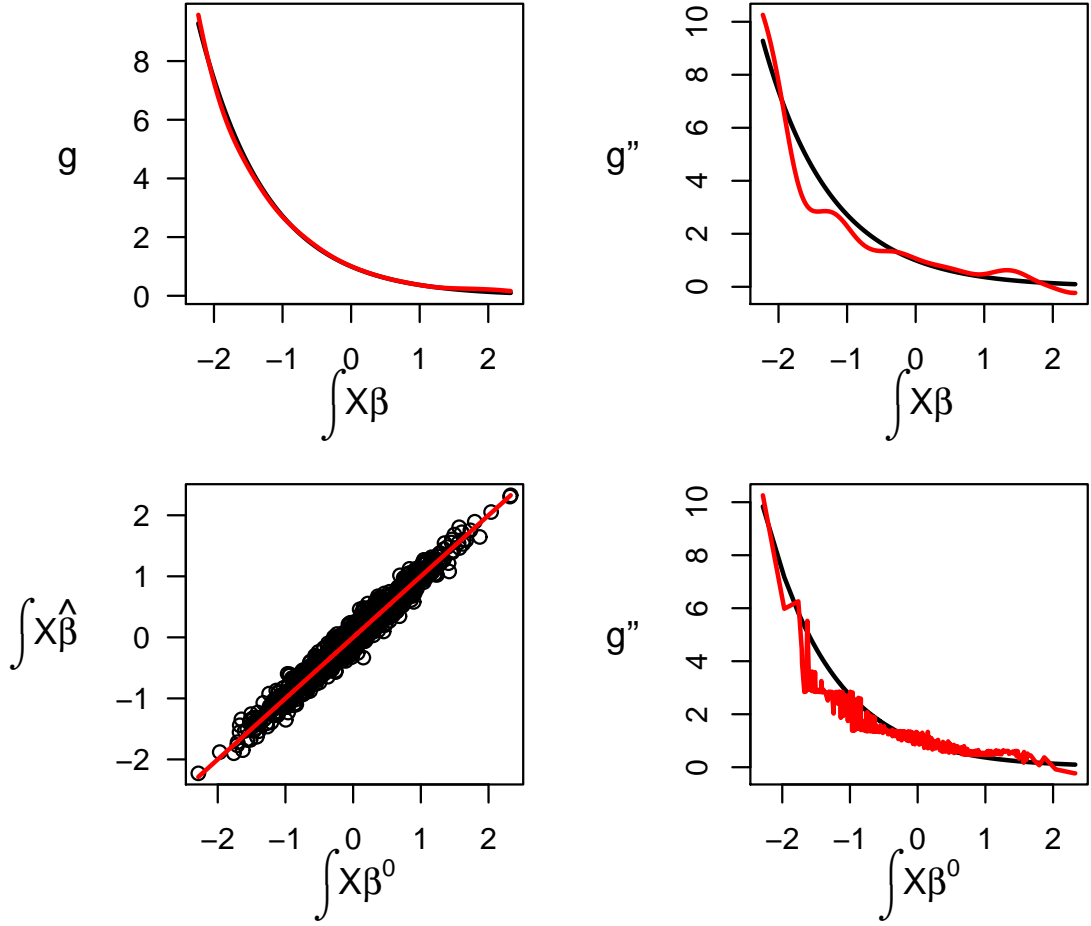}\\
  \caption{Example results using $g\left(s\right) = e^{-s}$. The top-left and right panel are the plots of $g$ and $g''$ over $1000$ equally-spaced grid points, while the lower and upper bound are the minimum and maximum of $\int_0^1 X\left(t\right)\hat{\beta}\left(t\right)\mathrm{d}t$. The bottom-right panel gives $g''\left(\int_0^1 X\left(t\right)\hat{\beta}\left(t\right)\mathrm{d}t\right)$ plotted against the true $\int_0^1 X\left(t\right)\beta^0\left(t\right)\mathrm{d}t$, providing a visual representation of the error we control in Theorem $6$. In all plots, the true curve is given by the black line, while the red line is the estimated curve. The bottom-left panel is the plot of $\int_0^1 X\left(t\right)\hat{\beta}\left(t\right)\mathrm{d}t$ versus $\int_0^1 X\left(t\right)\beta^0\left(t\right)\mathrm{d}t$, while the red line is $y = x$.}\label{}
 \end{center}
\end{figure}
\begin{figure}
\begin{center}
  \includegraphics[width=15cm]{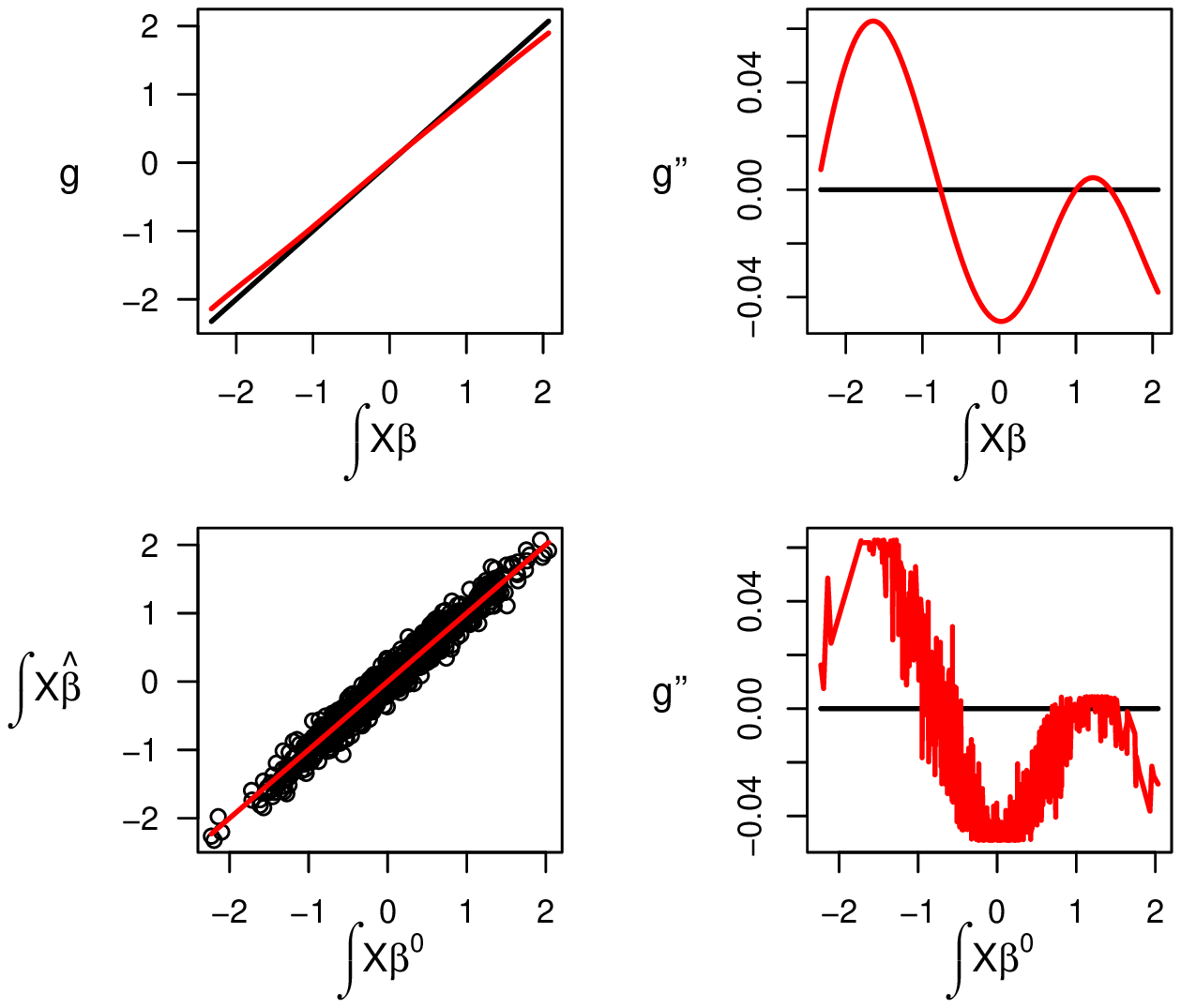}\\
  \caption{Example results using $g\left(s\right) = s$. The top-left and right panel are the plots of $g$ and $g''$ over $1000$ equally-spaced grid points, while the lower and upper bound are the minimum and maximum of $\int_0^1 X\left(t\right)\hat{\beta}\left(t\right)\mathrm{d}t$. The bottom-right panel gives $g''\left(\int_0^1 X\left(t\right)\hat{\beta}\left(t\right)\mathrm{d}t\right)$ plotted against the true $\int_0^1 X\left(t\right)\beta^0\left(t\right)\mathrm{d}t$, providing a visual representation of the error we control in Theorem $6$. In all plots, the true curve is given by the black line,, while the red line is the estimated curve. The bottom-left panel is the plot of $\int_0^1 X\left(t\right)\hat{\beta}\left(t\right)\mathrm{d}t$ versus $\int_0^1 X\left(t\right)\beta^0\left(t\right)\mathrm{d}t$, while the red line is $y = x$.}\label{}
 \end{center}
\end{figure}

\section{Ecological Data}
\subsection{Model Formulation}
Examining our ecological questions, the purpose of estimating the second derivative of the link function in a functional single index model is to figure out whether the link function $g$ is convex or concave. Then, we can answer the question: in which environment, constant or varying, the plant will grow better. We apply our nested estimation method to plant growth dataset. In this dataset, there are several variables:
\begin{enumerate}
\item $\tt{logarea.t1, logarea.t0:}$ the plant's logarithm of area at time $t0$ and $t1$, where $t0$ is the observation start time and $t1$ is the end time. A relatively large quantity indicates a high growth rate of the plant at that time.
\item $\tt{W}$: a measure of plant competition. Taken to be a scalar covariate.
\item $\tt{p.00 - p.36}$: discrete aggregated temporal record of precipitation, denoted as $p\left(s\right)$.
\item $\tt{t.00 - t.36}$: discrete aggregated temporal record of temperature, denoted as $t\left(s\right)$.
\end{enumerate}
The precipitation and temperature histories are modeled as two covariate functions. Assume that the response variable is $logarea.t1 - logarea.t0$, a Functional Single Index model is:
\begin{align}
logarea.t1 - logarea.t0 = g\left(\alpha\cdot W + \int p\beta_1 + \int t\beta_2\right),\nonumber
\end{align}
where the coefficient $\alpha$, the functions $g$, $\beta_1$ and $\beta_2$ need to be estimated.

\subsection{Results}
We used three different starting values: linear, equal and random. For each starting point, we selected the curve with minimum GCV value. Then, we selected the estimate with the minimum GCV value among all starting points. The random starting point is selected. The plot of the estimated $g$, $g''$, and the coefficient functions $\beta_1$ and $\beta_2$ is in the Figure $3$. Since the estimated $g''$ is always negative, the link function $g$ is concave. We could conclude that the species will grow better with a constant environment.
\begin{figure}
\begin{center}
  \includegraphics[width=15cm]{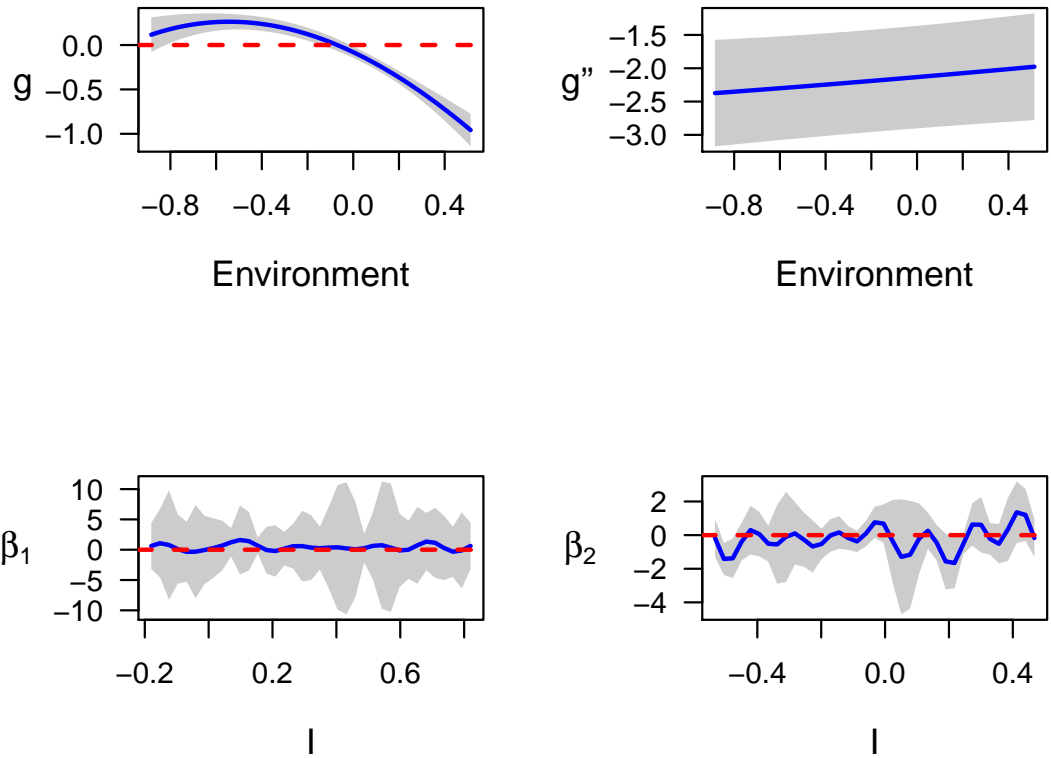}\\
  \caption{The top-left panel is the estimate of $g$, while the top-right panel is the estimate of $g''$. The bottom panel are the plots of the estimated $\beta_1$ and $\beta_2$.}\label{}
 \end{center}
\end{figure}

\section{Conclusion}
To answer the ecological question, we need to figure out the convexity or concavity of the link function $g$, or equivalently, find whether the second derivative is positive or negative. In this paper, we used the local quadratic method to approximate the link function $g$, and estimated the curvature of $g$ and the coefficient function $\beta^0$ by a nested optimization procedure. Under some assumptions, we showed that the coefficient function $\beta^0$ could be estimated root-$n$ consistently. In addition, the rate of convergence of the curvature $g''$ is $\frac{1}{n}\sum\limits_{i=1}^n\mathrm{E}\left[\hat{g}''\left(\int X_i\hat{\beta}\right) - g''\left(\int X_i\beta^0\right)\right]^2 = O\left(h_n^4 + \frac{1}{nh_n^4}\right)$.\\
In the simulation study, we used three different link functions, convex, concave and neither convex nor concave. While we derive convergence rates for the curvature of $g$,  our simulation results demonstrate the numerical challenges that accompany Functional Single Index models.  We can estimate $g$ fairly well, but our estimates of $g''$ are sensitive to the choice of initial condition, requiring considerable care in optimization. We expect that these numerical challenges are specific to the estimators employed, but they suggest that alternative means for the influence of environmental variability on plant growth are warranted.

\clearpage
\appendix
\section{Bernstein's Inequality}
\textbf{Bernstein's Inequality.} Let $Y_{1n}, \cdots, Y_{nn}$ be independent random variables with means $0$ and bounded ranges, that is $\left|Y_{in}\right| \leq c_n$. Write $\sigma_{in}^2$ for the variance of $Y_{in}$. Suppose $V_n \geq \sigma_{1n}^2 + \cdots + \sigma_{nn}^2$. Then for each $\eta_n > 0$,
\begin{align}
\mathrm{P}\left(\left|Y_{1n} + \cdots + Y_{nn}\right| > \eta_n\right) \leq \exp\left[-\frac{\eta_n^2}{2\left(V_n + \frac{1}{3}c_n\eta_n\right)}\right].\nonumber
\end{align}

\section{Lemmas for Theorem $6$}
\begin{lemma}
7 If $nh_n^6 \rightarrow \infty$, $nh_n^8 \rightarrow 0$ and $\frac{nh_n^{3+\frac{3}{m-1}}}{-\log h_n} \rightarrow \infty$, then we have
\begin{align}
\mathrm{E}\left[\hat{g}''\left(\int X_j\hat{\beta}\right) - \bar{g}''\left(\int X_j\hat{\beta}\right)\right]^2  = O\left(\frac{1}{nh_n^4}\right).\nonumber
\end{align}
\end{lemma}
\begin{proof}
Denote $\bm{\epsilon} = \left(\epsilon_1,\cdots,\epsilon_n\right)^{\top}$. The term could be expressed as
\begin{align}
&\mathrm{E}\left[\hat{g}''\left(\int X_j\hat{\beta}\right) - \bar{g}''\left(\int X_j\hat{\beta}\right)\right]^2 \nonumber \\
& \hspace{.4cm} = \mathrm{E}\left[S_2\left(\hat{\beta};j\right)\bm{\epsilon}\right]^2 \nonumber \\
& \hspace{.4cm} = \mathrm{E}\left[\frac{2T_j^0\left(\hat{\beta}\right)\sum\limits_{i=1}^n\left(\int X_i\hat{\beta}-\int X_j\hat{\beta}\right)^2K\left(\frac{\int X_i\hat{\beta}-\int X_j\hat{\beta}}{h_n}\right)\epsilon_i-2T_j^2\left(\hat{\beta}\right)\sum\limits_{i=1}^nK\left(\frac{\int X_i\hat{\beta}-\int X_j\hat{\beta}}{h_n}\right)\epsilon_i}{T_j^0\left(\hat{\beta}\right)T_j^4\left(\hat{\beta}\right) - \left(T_j^2\left(\hat{\beta}\right)\right)^2}\right]^2 \nonumber \\
& \hspace{.4cm} \leq \mathrm{E}\left[\frac{2T_j^0\left(\hat{\beta}\right)\sum\limits_{i=1}^n\left(\int X_i\hat{\beta}-\int X_j\hat{\beta}\right)^2K\left(\frac{\int X_i\hat{\beta}-\int X_j\hat{\beta}}{h_n}\right)\epsilon_i}{T_j^0\left(\hat{\beta}\right)T_j^4\left(\hat{\beta}\right) - \left(T_j^2\left(\hat{\beta}\right)\right)^2}\right]^2 + \mathrm{E}\left[\frac{2T_j^2\left(\hat{\beta}\right)\sum\limits_{i=1}^nK\left(\frac{\int X_i\hat{\beta}-\int X_j\hat{\beta}}{h_n}\right)\epsilon_i}{T_j^0\left(\hat{\beta}\right)T_j^4\left(\hat{\beta}\right) - \left(T_j^2\left(\hat{\beta}\right)\right)^2}\right]^2. \nonumber 
\end{align}
We bound the next two terms as:
\begin{align}
& \mathrm{E}\left[\frac{T_j^0\left(\hat{\beta}\right)\sum\limits_{i=1}^n\left(\int X_i\hat{\beta}-\int X_j\hat{\beta}\right)^2K\left(\frac{\int X_i\hat{\beta}-\int X_j\hat{\beta}}{h_n}\right)\epsilon_i}{T_j^0\left(\hat{\beta}\right)T_j^4\left(\hat{\beta}\right) - \left(T_j^2\left(\hat{\beta}\right)\right)^2}\right]^2 \nonumber \\
& \hspace{.4cm} = \mathrm{E}\left[\frac{1}{h_n^5}\cdot\frac{\frac{1}{nh_nf\left(u\left|\hat{\beta}\right.\right)}T_j^0\left(\hat{\beta}\right)\cdot\frac{1}{n}\sum\limits_{i=1}^n\left(\int X_i\hat{\beta}-\int X_j\hat{\beta}\right)^2K\left(\frac{\int X_i\hat{\beta}-\int X_j\hat{\beta}}{h_n}\right)\epsilon_i}{\frac{1}{nh_nf\left(u\left|\hat{\beta}\right.\right)}T_j^0\left(\hat{\beta}\right)\cdot\frac{1}{nh_n^5f\left(u\left|\hat{\beta}\right.\right)}T_j^4\left(\hat{\beta}\right) - \left(\frac{1}{nh_n^3f\left(u\left|\hat{\beta}\right.\right)}T_j^2\left(\hat{\beta}\right)\right)^2}\right]^2 \nonumber \\
& \hspace{.4cm} = \left(\frac{1}{h_n^5}\cdot\frac{1}{\mu_4\left(K\right)-\left(\mu_2\left(K\right)\right)^2}\cdot \frac{1}{f\left(u\left|\hat{\beta}\right.\right)}\right)^2 \cdot \frac{1}{n} \left[\int \left(z-u\right)^2K\left(\frac{z-u}{h_n}\right)f\left(z\left|\hat{\beta}\right.\right)\mathrm{d}z\right]^2 \cdot \mathrm{E}\left(\epsilon^2\right) \nonumber \\ 
& \hspace{.4cm} \doteq M_1 \cdot \frac{1}{nh_n^{10}f^2\left(u\left|\hat{\beta}\right.\right)}\left[\int \left(h_nm\right)^2K\left(m\right)f\left(u+h_nm\left|\hat{\beta}\right.\right)h_n\mathrm{d}m\right]^2 \cdot \mathrm{E}\left(\epsilon^2\right) \nonumber \\
& \hspace{.4cm} \sim  M_1 \cdot \frac{1}{nh_n^{10}f^2\left(u\left|\hat{\beta}\right.\right)}\left[h_n^3f\left(u\left|\hat{\beta}\right.\right) \int m^2K\left(m\right)\mathrm{d}m\right]^2 \cdot \mathrm{E}\left(\epsilon^2\right) \nonumber \\
& \hspace{.4cm} \sim \frac{1}{nh_n^4}, \nonumber
\end{align}
and
\begin{align}
& \mathrm{E}\left[\frac{T_j^2\left(\hat{\beta}\right)\sum\limits_{i=1}^nK\left(\frac{\int X_i\hat{\beta}-\int X_j\hat{\beta}}{h_n}\right)\epsilon_i}{T_j^0\left(\hat{\beta}\right)T_j^4\left(\hat{\beta}\right) - \left(T_j^2\left(\hat{\beta}\right)\right)^2}\right]^2 \nonumber\\
& \hspace{.4cm} = \mathrm{E}\left[\frac{1}{h_n^3}\cdot\frac{\frac{1}{nh_n^3\left(u\right)}T_j^2\left(\hat{\beta}\right)\cdot \frac{1}{n}\sum\limits_{i=1}^nK\left(\frac{\int X_i\hat{\beta}-\int X_j\hat{\beta}}{h_n}\right)\epsilon_i}{\frac{1}{nh_nf\left(u\left|\hat{\beta}\right.\right)}T_j^0\left(\hat{\beta}\right)\cdot\frac{1}{nh_n^5f\left(u\left|\hat{\beta}\right.\right)}T_j^4\left(\hat{\beta}\right) - \left(\frac{1}{nh_n^3f\left(u\right)}T_j^2\left(\hat{\beta}\right)\right)^2}\right]^2 \nonumber \\
& \hspace{.4cm} = \left(\frac{1}{h_n^3}\cdot\frac{\mu_2\left(K\right)}{\mu_4\left(K\right)-\left(\mu_2\left(K\right)\right)^2}\cdot \frac{1}{f\left(u\left|\hat{\beta}\right.\right)}\right)^2 \cdot \frac{1}{n} \left[\int K\left(\frac{z-u}{h_n}\right)f\left(z\left|\hat{\beta}\right.\right)\mathrm{d}z\right]^2 \cdot \mathrm{E}\left(\epsilon^2\right) \nonumber \\ 
& \hspace{.4cm} \doteq  M_2 \cdot \frac{1}{nh_n^{6}f^2\left(u\left|\hat{\beta}\right.\right)}\left[\int K\left(m\right)f\left(u+h_nm\left|\hat{\beta}\right.\right)h_n\mathrm{d}m\right]^2 \cdot \mathrm{E}\left(\epsilon^2\right) \nonumber \\
& \hspace{.4cm} \sim M_2 \cdot \frac{1}{nh_n^{6}f^2\left(u\left|\hat{\beta}\right.\right)}\left[h_nf\left(u\left|\hat{\beta}\right.\right) \int K\left(m\right)\mathrm{d}m\right]^2 \cdot \mathrm{E}\left(\epsilon^2\right) \nonumber \\
& \hspace{.4cm} \sim \frac{1}{nh_n^4}. \nonumber
\end{align}
where $M_1$, $M_2$ are constants and $u = \int X_j\hat{\beta}$.\\
Therefore, we have
\begin{align}
\mathrm{E}\left[\hat{g}''\left(\int X_j\hat{\beta}\right) - \bar{g}''\left(\int X_j\hat{\beta}\right)\right]^2 = O\left(\frac{1}{nh_n^4}\right).\nonumber
\end{align}
\end{proof}
\begin{lemma}
8 If $nh_n^6 \rightarrow \infty$, $nh_n^8 \rightarrow 0$ and $\frac{nh_n^{3+\frac{3}{m-1}}}{-\log h_n} \rightarrow \infty$, then we have
\begin{align}
\mathrm{E}\left[\bar{g}''\left(\int X_j\hat{\beta}\right) - g''\left(\int X_j\hat{\beta}\right)\right]^2 = O\left(h_n^4 + \frac{1}{nh_n^4}\right).\nonumber
\end{align}
\end{lemma}
\begin{proof}
Calculate
\begin{align}
&\bar{g}''\left(\int X_j\hat{\beta}\right) - g''\left(\int X_j\hat{\beta}\right) \nonumber \\
& \hspace{.4cm} = S_2\left(\hat{\beta};j\right)\bm{g} - g''\left(\int X_j\hat{\beta}\right) \nonumber\\
& \hspace{.4cm} =  \frac{2T_j^0\left(\hat{\beta}\right)\sum\limits_{i=1}^n\left(\int X_i\hat{\beta}-\int X_j\hat{\beta}\right)^2K\left(\frac{\int X_i\hat{\beta}-\int X_j\hat{\beta}}{h_n}\right)g\left(\int X_i\beta^0\right)}{T_j^0\left(\hat{\beta}\right)T_j^4\left(\hat{\beta}\right) - \left(T_j^2\left(\hat{\beta}\right)\right)^2} \nonumber \\
& \hspace{.8cm} - \frac{-2T_j^2\left(\hat{\beta}\right)\sum\limits_{i=1}^nK\left(\frac{\int X_i\hat{\beta}-\int X_j\hat{\beta}}{h_n}\right)g\left(\int X_i\beta^0\right)}{T_j^0\left(\hat{\beta}\right)T_j^4\left(\hat{\beta}\right) - \left(T_j^2\left(\hat{\beta}\right)\right)^2} - g''\left(\int X_j\hat{\beta}\right). \nonumber
\end{align}
Write
\begin{align}
g\left(\int X_i\beta^0\right) = g\left(\int X_i\hat{\beta}\right) + \left(g\left(\int X_i\beta^0\right) - g\left(\int X_i\hat{\beta}\right)\right),\nonumber
\end{align}
for $i = 1,\cdots,n$, we decompose it into two terms $A$ and $B$, where $A$ is
\begin{align}
& \frac{2T_j^0\left(\hat{\beta}\right)\sum\limits_{i=1}^n\left(\int X_i\hat{\beta}-\int X_j\hat{\beta}\right)^2K\left(\frac{\int X_i\hat{\beta}-\int X_j\hat{\beta}}{h_n}\right)g\left(\int X_i\hat{\beta}\right)}{T_j^0\left(\hat{\beta}\right)T_j^4\left(\hat{\beta}\right) - \left(T_j^2\left(\hat{\beta}\right)\right)^2} \nonumber \\
& \hspace{.4cm} - \frac{2T_j^2\left(\hat{\beta}\right)\sum\limits_{i=1}^nK\left(\frac{\int X_i\hat{\beta}-\int X_j\hat{\beta}}{h_n}\right)g\left(\int X_i\hat{\beta}\right)}{T_j^0\left(\hat{\beta}\right)T_j^4\left(\hat{\beta}\right) - \left(T_j^2\left(\hat{\beta}\right)\right)^2} - g''\left(\int X_j\hat{\beta}\right), \nonumber
\end{align}
and $B$ is
\begin{align}
\frac{\sum\limits_{i=1}^n\left[2T_j^0\left(\hat{\beta}\right)\left(\int X_i\hat{\beta}-\int X_j\hat{\beta}\right)^2K\left(\frac{\int X_i\hat{\beta}-\int X_j\hat{\beta}}{h_n}\right)-2T_j^2\left(\hat{\beta}\right)K\left(\frac{\int X_i\hat{\beta}-\int X_j\hat{\beta}}{h_n}\right)\right]\left(g\left(\int X_i\beta^0\right) - g\left(\int X_i\hat{\beta}\right)\right)}{T_j^0\left(\hat{\beta}\right)T_j^4\left(\hat{\beta}\right) - \left(T_j^2\left(\hat{\beta}\right)\right)^2} .\nonumber
\end{align}
Denote $K_{ij} \doteq K\left(\frac{\int X_i\hat{\beta}-\int X_j\hat{\beta}}{h_n}\right)$, the term $A$ can be calculated as
\begin{align}
& \frac{2T_j^0\left(\hat{\beta}\right)\sum\limits_{i=1}^n\left(\int X_i\hat{\beta}-\int X_j\hat{\beta}\right)^2K_{ij}g\left(\int X_i\hat{\beta}\right)-2T_j^2\left(\hat{\beta}\right)\sum\limits_{i=1}^nK_{ij}g\left(\int X_i\hat{\beta}\right)}{T_j^0\left(\hat{\beta}\right)T_j^4\left(\hat{\beta}\right) - \left(T_j^2\left(\hat{\beta}\right)\right)^2} - g''\left(\int X_j\hat{\beta}\right)\nonumber \\
& \hspace{.4cm} = \frac{2T_j^0\left(\hat{\beta}\right)\sum\limits_{i=1}^n\left(\int X_i\hat{\beta}-\int X_j\hat{\beta}\right)^2K_{ij}g\left(\int X_i\hat{\beta}\right)}{T_j^0\left(\hat{\beta}\right)T_j^4\left(\hat{\beta}\right) - \left(T_j^2\left(\hat{\beta}\right)\right)^2} - 
\frac{2T_j^2\left(\hat{\beta}\right)\sum\limits_{i=1}^nK_{ij}g\left(\int X_i\hat{\beta}\right)}{T_j^0\left(\hat{\beta}\right)T_j^4\left(\hat{\beta}\right) - \left(T_j^2\left(\hat{\beta}\right)\right)^2} - g''\left(\int X_j\hat{\beta}\right),\nonumber
\end{align}
where
\begin{align}
&  \frac{2T_j^0\left(\hat{\beta}\right)\sum\limits_{i=1}^n\left(\int X_i\hat{\beta}-\int X_j\hat{\beta}\right)^2K\left(\frac{\int X_i\hat{\beta}-\int X_j\hat{\beta}}{h_n}\right)g\left(\int X_i\hat{\beta}\right)}{T_j^0\left(\hat{\beta}\right)T_j^4\left(\hat{\beta}\right) - \left(T_j^2\left(\hat{\beta}\right)\right)^2} \nonumber\\
& \hspace{.4cm} =  \frac{2T_j^0\left(\hat{\beta}\right)\sum\limits_{i=1}^n\left(\int X_i\hat{\beta}-u\right)^2K_{ij}\left[g\left(u\right)+g'\left(u\right)\left(\int X_i\hat{\beta}-u\right)+\frac{1}{2}g''\left(u\right)\left(\int X_i\hat{\beta}-u\right)^2 + O\left(\int X_i\hat{\beta}-u\right)^4\right]}{T_j^0\left(\hat{\beta}\right)T_j^4\left(\hat{\beta}\right) - \left(T_j^2\left(\hat{\beta}\right)\right)^2} \nonumber \\
& \hspace{.4cm} = \frac{2T_j^0\left(\hat{\beta}\right)\left(T_j^2\left(\hat{\beta}\right)g\left(u\right)+T_j^3\left(\hat{\beta}\right)g'\left(u\right)+\frac{1}{2}T_j^4\left(\hat{\beta}\right)g''\left(u\right)\right)}{T_j^0\left(\hat{\beta}\right)T_j^4\left(\hat{\beta}\right) - \left(T_j^2\left(\hat{\beta}\right)\right)^2} + O\left(h_n^2\right) \nonumber \\
& \hspace{.4cm} = \frac{2T_j^0\left(\hat{\beta}\right)T_j^2\left(\hat{\beta}\right)g\left(\int X_j\hat{\beta}\right)+T_j^0\left(\hat{\beta}\right)T_j^4\left(\hat{\beta}\right)g''\left(\int X_j\hat{\beta}\right)}{T_j^0\left(\hat{\beta}\right)T_j^4\left(\hat{\beta}\right) - \left(T_j^2\left(\hat{\beta}\right)\right)^2} + O\left(h_n^2\right), \nonumber
\end{align}
and
\begin{align}
& \frac{2T_j^2\left(\hat{\beta}\right)\sum\limits_{i=1}^nK\left(\frac{\int X_i\hat{\beta}-\int X_j\hat{\beta}}{h_n}\right)g\left(\int X_i\hat{\beta}\right)}{T_j^0\left(\hat{\beta}\right)T_j^4\left(\hat{\beta}\right) - \left(T_j^2\left(\hat{\beta}\right)\right)^2}\nonumber \\
& \hspace{.4cm} =  \frac{2T_j^2\left(\hat{\beta}\right)\sum\limits_{i=1}^nK\left(\frac{\int X_i\hat{\beta}-u}{h_n}\right)\left[g\left(u\right)+g'\left(u\right)\left(\int X_i\hat{\beta}-u\right)+\frac{1}{2}g''\left(u\right)\left(\int X_i\hat{\beta}-u\right)^2 +O\left(\int X_i\hat{\beta}-u\right)^4\right]}{T_j^0\left(\hat{\beta}\right)T_j^4\left(\hat{\beta}\right) - \left(T_j^2\left(\hat{\beta}\right)\right)^2} \nonumber \\
& \hspace{.4cm} =  \frac{2T_j^2\left(\hat{\beta}\right)\left(T_j^0\left(\hat{\beta}\right)g\left(u\right)+T_j^1\left(\hat{\beta}\right)g'\left(u\right)+\frac{1}{2}T_j^2\left(\hat{\beta}\right)g''\left(u\right)\right)}{T_j^0\left(\hat{\beta}\right)T_j^4\left(\hat{\beta}\right) - \left(T_j^2\left(\hat{\beta}\right)\right)^2} + O\left(h_n^2\right)\nonumber \\
& \hspace{.4cm} =  \frac{2T_j^2\left(\hat{\beta}\right)T_j^0\left(\hat{\beta}\right)g\left(\int X_j\hat{\beta}\right)+\left(T_j^2\left(\hat{\beta}\right)\right)^2g''\left(\int X_j\hat{\beta}\right)}{T_j^0\left(\hat{\beta}\right)T_j^4\left(\hat{\beta}\right) - \left(T_j^2\left(\hat{\beta}\right)\right)^2} + O\left(h_n^2\right).\nonumber
\end{align}
Therefore, the term $A$ is
\begin{align}
&\frac{2T_j^0\left(\hat{\beta}\right)T_j^2\left(\hat{\beta}\right)g\left(\int X_j\hat{\beta}\right)+T_j^0\left(\hat{\beta}\right)T_j^4\left(\hat{\beta}\right)g''\left(\int X_j\hat{\beta}\right)}{T_j^0\left(\hat{\beta}\right)T_j^4\left(\hat{\beta}\right) - \left(T_j^2\left(\hat{\beta}\right)\right)^2} \nonumber \\
& \hspace{.8cm} - \frac{2T_j^2\left(\hat{\beta}\right)T_j^0\left(\hat{\beta}\right)g\left(\int X_j\hat{\beta}\right)+\left(T_j^2\left(\hat{\beta}\right)\right)^2g''\left(\int X_j\hat{\beta}\right)}{T_j^0\left(\hat{\beta}\right)T_j^4\left(\hat{\beta}\right) - \left(T_j^2\left(\hat{\beta}\right)\right)^2} - g''\left(\int X_j\hat{\beta}\right) + O\left(h_n^2\right) \nonumber \\
& \hspace{.4cm} = g''\left(\int X_j\hat{\beta}\right) - g''\left(\int X_j\hat{\beta}\right) + O\left(h_n^2\right) \nonumber \\
& \hspace{.4cm} = O\left(h_n^2\right).\nonumber
\end{align}
The term $B$ is bounded by
\begin{align}
\left|B\right| \leq & \left|\frac{\sum\limits_{i=1}^n\left[2T_j^0\left(\hat{\beta}\right)\left(\int X_i\hat{\beta}-\int X_j\hat{\beta}\right)^2K\left(\frac{\int X_i\hat{\beta}-\int X_j\hat{\beta}}{h_n}\right)-2T_j^2\left(\hat{\beta}\right)K\left(\frac{\int X_i\hat{\beta}-\int X_j\hat{\beta}}{h_n}\right)\right]\left(E_1\frac{1}{\sqrt{n}}\right)}{T_j^0\left(\hat{\beta}\right)T_j^4\left(\hat{\beta}\right) - \left(T_j^2\left(\hat{\beta}\right)\right)^2}\right| \nonumber \\
=& \frac{\left(2T_j^0\left(\hat{\beta}\right)T_j^2\left(\hat{\beta}\right)-2T_j^0\left(\hat{\beta}\right)T_j^2\left(\hat{\beta}\right)\right)\left(E_1\frac{1}{\sqrt{n}}\right)}{T_j^0\left(\hat{\beta}\right)T_j^4\left(\hat{\beta}\right) - \left(T_j^2\left(\hat{\beta}\right)\right)^2} \nonumber \\
\sim & \frac{1}{\sqrt{n}h_n^2}, \nonumber
\end{align}
where $E_1$ is a constant.\\
Combining the terms $A$ and $B$, we have
\begin{align}
\mathrm{E}\left[\bar{g}''\left(\int X_j\hat{\beta}\right) - g''\left(\int X_j\hat{\beta}\right)\right]^2 = O\left(h_n^4 + \frac{1}{nh_n^4}\right).\nonumber
\end{align}
\end{proof}
\begin{lemma}
9 If $nh_n^6 \rightarrow \infty$, $nh_n^8 \rightarrow 0$ and $\frac{nh_n^{3+\frac{3}{m-1}}}{-\log h_n} \rightarrow \infty$, then we have
\begin{align}
\mathrm{E}\left[g''\left(\int X_j\hat{\beta}\right) - g''\left(\int X_j\beta^0\right)\right]^2 = O\left(\frac{1}{n}\right).\nonumber
\end{align}
\end{lemma}
\begin{proof}
By Theorem $5$, we have $\left(\int X_{i}\hat{\beta} - \int X_{i}\beta^0\right) = O\left(\frac{1}{\sqrt{n}}\right)$. Since $g''$ satisfies the Lipschitz condition, we can get
\begin{align}
\mathrm{E}\left[g''\left(\int X_j\hat{\beta}\right) - g''\left(\int X_j\beta^0\right)\right]^2 \leq K_3 \left(\int X_{i}\hat{\beta} - \int X_{i}\beta^0\right)^2 = O\left(\frac{1}{n}\right),\nonumber
\end{align}
where $K_3$ is a constant.
\end{proof}

\section{$10$-fold CV Results}
The simulation results using the $10$-fold cross-validation is:
\begin {table}[H]
\caption {Simulation results by $10$-fold CV using rescaled bandwidth $h\left\|\bm{c}\right\|$ and constraint $\hat{\beta}$.} \label{table5} 
\begin{center}
\begin{tabular}{ccccc|ccc|cccc}
\hline
\multirow{2}{*}{}&\multirow{2}{*}{} & \multicolumn{3}{c}{\textbf{g1}} &  \multicolumn{3}{c}{\textbf{g2}} & \multicolumn{3}{c}{\textbf{g3}} \\
\cline{3-11}
\textbf{Initial} & \textbf{n} & \textbf{RSE} & \textbf{RASE} & \textbf{RASE2} & \textbf{RSE} & \textbf{RASE} & \textbf{RASE2} & \textbf{RSE} & \textbf{RASE} & \textbf{RASE2} & \\
\hline
\multirow{2}{*}{True} & 100 & 0.1921 & 0.0868 & 1.2106 & 0.2052 & 0.0904 & 4.5719 & 0.1982 & 0.0690 & 0.7954\\
 & 1000& 0.0680 & 0.0290 & 0.4551 & 0.0620 & 0.0271 & 0.4016 & 0.0774 & 0.0243 & 0.2762 \\
\hline
\multirow{2}{*}{Linear} & 100 & 1.1383 & 0.3294 & 1.4303 & 1.3370 & 0.4245 & 1.3388 & 0.6995 & 0.1331 & 0.2692 \\
& 1000 & 1.9015 & 0.2232 & 0.5060 & 1.0419 & 0.3185 & 1.3652 & 0.5297 & 0.0852 & 0.0786 \\
\hline
\multirow{2}{*}{Equal} & 100 & 0.8486 & 0.2789 & 1.0731 & 0.8385 & 0.2509 & 1.4728 & 0.8349 & 0.1836 & 0.3349 \\
& 1000 & 0.8442 & 0.2977 & 1.0181 & 0.8500 & 0.2755 & 1.5396 & 0.8447 & 0.1827 & 0.1045 \\
\hline
\multirow{2}{*}{Random} & 100 & 1.2998 & 0.1339 & 0.5191 & 1.2086 & 0.1979 & 0.7160& 1.4296 & 0.0875 & 0.2362\\
& 1000 & 1.5607 & 0.1288 & 0.4209 & 1.3540 & 0.1687 & 0.4418 & 1.7622 & 0.0798 & 0.0744 
\\[1ex] \hline
\end{tabular}
\end{center}
\end{table}
\noindent For comparison purpose, we estimate the curvature with or without rescaling the bandwidth, where the results are in Table $4$.
\begin{table}
\caption {Simulation results of random starting value with rescaled and original bandwidths selected by $10$-fold cross-validation} \label{table1234} 
\begin{center}
\begin{tabular}{ccc|cc|ccc}
\hline
\multirow{2}{*}{} & \multicolumn{2}{c}{\textbf{g1}} &  \multicolumn{2}{c}{\textbf{g2}} & \multicolumn{2}{c}{\textbf{g3}} \\
\cline{2-7}
\textbf{n} & \textbf{original} & \textbf{rescaled} & \textbf{original} & \textbf{rescaled} &  \textbf{original} & \textbf{rescaled}  \\
\hline
\multirow{2}{*}{}  100 & 1.2710  &0.5191  & 2.2005&  0.7160  & 0.2768 & 0.2362 \\
 1000 & 1.0145 & 0.4209  & 2.2812 & 0.4418  & 0.0752 & 0.0744
 \\[1ex] \hline
\end{tabular}
\end{center}
\end{table}

\section{True Starting Value}
In Table $5$ and $6$, we compare the RASE$2$ results of true starting value with rescaled and original bandwidths.
\begin{table}
\caption {Simulation results of true starting value with rescaled and original bandwidths selected by $10$-fold cross-validation} \label{table123} 
\begin{center}
\begin{tabular}{ccc|cc|ccc}
\hline
\multirow{2}{*}{} & \multicolumn{2}{c}{\textbf{g1}} &  \multicolumn{2}{c}{\textbf{g2}} & \multicolumn{2}{c}{\textbf{g3}} \\
\cline{2-7}
\textbf{n} & \textbf{original} & \textbf{rescaled} & \textbf{original} & \textbf{rescaled} &  \textbf{original} & \textbf{rescaled}  \\
\hline
\multirow{2}{*}{}  100 & 5.9302  &1.9266  & 33.811& 7.9015  & 4.5630 & 0.9791 \\
 1000 & 2.4004 & 1.0660  & 2.2494 & 1.3416  & 1.2444 & 0.3558 
 \\[1ex] \hline
\end{tabular}
\end{center}
\end{table}
\begin {table}[H]
\caption {Simulation results of true starting value with rescaled and original bandwidths selected by GCV} \label{table4} 
\begin{center}
\begin{tabular}{ccc|cc|ccc}
\hline
\multirow{2}{*}{} & \multicolumn{2}{c}{\textbf{g1}} &  \multicolumn{2}{c}{\textbf{g2}} & \multicolumn{2}{c}{\textbf{g3}} \\
\cline{2-7}
\textbf{n} & \textbf{original} & \textbf{rescaled}  & \textbf{original} & \textbf{rescaled}  &  \textbf{original} & \textbf{rescaled} \\
\hline
\multirow{2}{*}{}  100 & 51.397  &7.3786  & 38.007 & 6.5712  & 50.816 &  5.9508 \\ 
 1000 & 9.0220 & 1.5416   & 8.9590 & 1.8170 & 7.8245 & 1.1493 
\\[1ex] \hline
\end{tabular}
\end{center}
\end{table}

\section{CV Values}
The CV values for both GCV and $10$-fold cross-validation are in Table $7$.
\begin {table}[H]
\caption {CV values} \label{table389} 
\begin{center}
\begin{tabular}{cccc|cc|ccc}
\hline
\multirow{2}{*}{}&\multirow{2}{*}{} & \multicolumn{2}{c}{\textbf{g1}} &  \multicolumn{2}{c}{\textbf{g2}} & \multicolumn{2}{c}{\textbf{g3}} \\
\cline{3-8}
\textbf{Initial} & \textbf{n} & \textbf{GCV} & \textbf{10-fold} & \textbf{GCV} & \textbf{10-fold} & \textbf{GCV} & \textbf{10-fold}  \\
\hline
\multirow{2}{*}{True} & 100 & 0.0398 & 3.8686 & 0.0399 & 3.3823 & 0.0377 & 1.4970 \\
 & 1000& 0.0401 & 0.9468 & 0.0402 & 0.7882 & 0.0400 & 0.4781\\
\hline
\multirow{2}{*}{Linear} & 100 & 0.1559 & 5.2259 & 0.2184 & 4.7373 & 0.0498 & 1.6148 \\
& 1000 & 0.0793 & 2.0295 & 0.1272 & 2.2896 & 0.0439 & 0.7265 \\
\hline
\multirow{2}{*}{Equal} & 100 & 0.1401 & 4.7640 & 0.1046 & 3.9670 & 0.0795 & 1.8530 \\
& 1000 & 0.1358 & 2.6925 & 0.1318 & 2.1995 & 0.0751 & 1.1081 \\
\hline
\multirow{2}{*}{Random} & 100 & 0.0611 & 4.2920 & 0.0957 & 3.9025 & 0.0459 & 1.6237\\
& 1000 & 0.0591 & 2.4463 & 0.0850 & 2.2702 & 0.0468 & 1.0243
\\[1ex] \hline
\end{tabular}
\end{center}
\end{table}

\bibliographystyle{chicago}
\clearpage
\bibliography{theoretical}

\end{document}